\documentclass[12 pt]{amsart}
\usepackage{amssymb,latexsym,amsmath,amscd,amsthm,graphicx, color}
\usepackage[all]{xy}
\usepackage{pgf,tikz}
\usepackage{mathrsfs}
\usepackage{cite}
\usetikzlibrary{arrows}
\definecolor{uuuuuu}{rgb}{0.26666666666666666,0.26666666666666666,0.26666666666666666}
\definecolor{xdxdff}{rgb}{0.49019607843137253,0.49019607843137253,1.}
\definecolor{ffqqqq}{rgb}{1.,0.,0.}

\definecolor{uuuuuu}{rgb}{0.26666666666666666,0.26666666666666666,0.26666666666666666}
\definecolor{qqwuqq}{rgb}{0.,0.39215686274509803,0.}
\definecolor{zzttqq}{rgb}{0.6,0.2,0.}
\definecolor{xdxdff}{rgb}{0.49019607843137253,0.49019607843137253,1.}
\definecolor{qqqqff}{rgb}{0.,0.,1.}
\definecolor{cqcqcq}{rgb}{0.7529411764705882,0.7529411764705882,0.7529411764705882}

\definecolor{uuuuuu}{rgb}{0.26666666666666666,0.26666666666666666,0.26666666666666666}
\definecolor{qqwuqq}{rgb}{0.,0.39215686274509803,0.}
\definecolor{zzttqq}{rgb}{0.6,0.2,0.}
\definecolor{xdxdff}{rgb}{0.49019607843137253,0.49019607843137253,1.}
\definecolor{qqqqff}{rgb}{0.,0.,1.}
\definecolor{cqcqcq}{rgb}{0.7529411764705882,0.7529411764705882,0.7529411764705882}

\theoremstyle{plain}

\newtheorem{theorem}[subsection]{Theorem}

\newtheorem{lemma}[subsection]{Lemma}

\newtheorem{prop}[subsection]{Proposition}
\theoremstyle{definition}

\newtheorem{remark}[subsection]{Remark}

\newtheorem{notation}[subsection]{Notation}

\newcommand{\uu}{\cup}
\newcommand{\ii}{\cap}

\newcommand{\sci}{\subset}
\newcommand{\es}{\emptyset}
\newcommand{\set}[1]{\{#1\}}

\newcommand{\ga}{\alpha}
\newcommand{\gb}{\beta}

\newcommand{\gd}{\delta}
\renewcommand{\gg}{\gamma}

\newcommand{\gs}{\sigma}
\newcommand{\gt}{\tau}

\newcommand{\tit}{\textit}

\newcommand{\D}[1]{\mathbb{#1}}
\newcommand{\F}[1]{\mathfrak{#1}}

\newcommand{\te}{\text}

\newcommand{\ol}{\overline}
\newcommand{\ul}{\underline}

\newcommand{\tri}{\triangle}
\begin{document}

To appear, Topology Proceedings
\title{Quantization for uniform distributions of Cantor dusts on $\mathbb{R}^2$}

\author{Do\u gan \c C\"omez}
\address{Department of Mathematics \\
408E24 Minard Hall,
North Dakota State University
\\
Fargo, ND 58108-6050, USA.}
\email{Dogan.Comez@ndsu.edu}

\author{ Mrinal Kanti Roychowdhury}
\address{School of Mathematical and Statistical Sciences\\
University of Texas Rio Grande Valley\\
1201 West University Drive\\
Edinburg, TX 78539-2999, USA.}
\email{mrinal.roychowdhury@utrgv.edu}
\thanks{The research of the second author was supported by U.S. National Security Agency (NSA) Grant H98230-14-1-0320}

\date{}
\maketitle

\pagestyle{myheadings}\markboth{Do\u gan \c C\"omez and Mrinal Kanti Roychowdhury}{Quantization for uniform distributions of Cantor dusts on $\mathbb{R}^2$}

\begin{abstract}  Let $P$ be a Borel probability measure on $\mathbb R^2$ supported by the Cantor dusts generated by a set of $4^u,\ u\geq 1$, contractive similarity mappings satisfying the strong separation condition.  For this probability measure, we determine the optimal sets of $n$-means and the $n$th quantization errors for all $n\geq 2$. In addition, it is shown that though the quantization dimension of the measure $P$ is known, the quantization coefficient for $P$ does not exist.

\end{abstract}

\section{Introduction}

Quantization for a probability distribution is the process of estimating it by a discrete probability that assumes only a finite number of levels in its support. For an in-depth analysis of quantization of probability measures  one may consult the excellent source by Graf-Luschgy  \cite{GL1} and the sources \cite{AW, GKL, GL, GN,  Z} to name a few.  In this paper, we are interested in the quantization of a particular type of continuous singular self similar probability measures.

Let $\D R^d$ denote the $d$-dimensional Euclidean space with the Euclidean norm $\Vert \ \Vert . $  For any $d\geq 1$ and $n\in \D N$, the $n$th \textit{quantization error} for a Borel probability measure $P$ on $\D R^d$ is defined by
\begin{equation*} \label{eq1} V_n:=V_n(P)=\inf \Big\{\int \min_{a\in\alpha} \|x-a\|^2 dP(x) : \alpha \subset \mathbb R^d, \text{ card}(\alpha) \leq n \Big\}.\end{equation*}
If $\int \| x\|^2 dP(x)<\infty$, then there is some set $\alpha$ for
which the infimum is achieved \cite{GL, GL1, GKL}. Such a set $\ga$ for which the infimum occurs and contains no more than $n$ points is called an \tit{optimal set of $n$-means}. The elements in an optimal set of $n$-means are called \tit{optimal quantizers}. If $\ga$ is a finite set, in general, the error $\int \min_{a \in \ga} \|x-a\|^2 dP(x)$ is often referred to as the \tit{distortion error} for $\ga$, and is denoted by $V(P; \ga)$. Then, the $n$th {\it quantization error for $P$} is defined by $V_n:=V_n(P)=\inf\set{V(P; \ga) :\alpha \subset \mathbb R^d, \text{ card}(\alpha) \leq n}$. It is known that for a continuous probability measure $P$ an optimal set of $n$-means always has exactly $n$ elements  \cite{GL1}.  Naturally, $\lim_n V_n(P)=0; $ hence, one would like to know more about the rate of convergence. To know some recent results in the direction of optimal sets of $n$-means and the $n$th quantization errors one can see \cite{CR, DR, R1, R2, RR}.  The numbers
\[\ul D(P):=\liminf_{n\to \infty}  \frac{2\log n}{-\log V_n(P)}, \te{ and } \ol D(P):=\limsup_{n\to \infty} \frac{2\log n}{-\log V_n(P)}, \]
are, respectively, called the \tit{lower} and \tit{upper quantization dimensions} of the probability measure $P$. If $\ul D(P)=\ol D (P)$, the common value is called the \tit{quantization dimension} of $P$ and is denoted by $D(P)$. Quantization dimension measures the speed at which the specified measure of the error tends to zero as $n$ approaches to infinity. Quantization dimension is also connected with other dimensions of dynamical systems, for details one can see \cite{GL1}.  More accurate information about the asymptotics of the quantization error is provided by the {\it quantization coefficients.}  For any $s\in (0, +\infty)$, the numbers $\liminf_n n^{\frac 2 s} V_n(P)$ and $\limsup_n n^{\frac 2s} V_n(P)$ are, respectively, called the $s$-dimensional \tit{lower} and \tit{upper quantization coefficients} for $P$. If the $s$-dimensional lower and upper quantization coefficients for $P$ are finite and positive, then $s$ coincides with the quantization dimension of $P$.

Given a finite set $\ga\sci \D R^d$, the Voronoi region generated by $a\in \ga$ is defined by
\[M(a|\ga)=\set{x \in \D R^d : \|x-a\|=\min_{b \in \ga}\|x-b\|},\]
i.e., the Voronoi region generated by $a\in \ga$ is the set of all elements in $\D R^d$ which are nearest to $a . $  The family of sets $\set{M(a|\ga) : a \in \ga}$ is called the \tit{Voronoi diagram} or \tit{Voronoi tessellation} of $\D R^d$ with respect to $\ga$. The generator $a\in \ga$ is called the centroid of its own Voronoi region with respect to the probability distribution $P$, if \begin{align*}
a=\frac{1}{P(M(a|\ga))}\int_{M(a|\ga)} x dP=\frac{\int_{M(a|\ga)} x dP}{\int_{M(a|\ga)} dP}.
\end{align*}
A Voronoi tessellation is called a {\it centroidal Voronoi tessellation} (CVT) if the generators are the centroids of their own Voronoi regions with respect to the measure $P.$ The following proposition provides further information on the Voronoi regions generated by an optimal set of $n$-means.
\begin{prop}(\cite{GG, GL1}) \label{prop10}
Let $\ga$ be an optimal set of $n$-means, $a \in \ga$, and $M(a|\ga)$ be the Voronoi region generated by $a\in \ga$, i.e.,
$M(a|\ga)=\{x \in \mathbb R^d : \|x-a\|=\min_{b \in \ga} \|x-b\|\}.$
Then, for every $a \in\ga$,
$(i)$ $P(M(a|\ga))>0$, $(ii)$ $ P(\partial M(a|\ga))=0$, $(iii)$ $a=E(X : X \in M(a|\ga))$, and $(iv)$ $P$-almost surely the set $\set{M(a|\ga) : a \in \ga}$ forms a Voronoi partition of $\D R^d$.
\end{prop}
\begin{remark} For a Borel probability measure $P$ on $\D R^d$, an optimal set of $n$-means forms a centroidal Voronoi tessellation of $\D R^d$; however, the converse is not true in general \cite{DFG, GG, R2}.  Also, the optimal quantizers are the centroids of their own Voronoi regions \cite{GL1}.
\end{remark}

Let $C$ be the standard Cantor set generated by the similarity mappings $U_1(x)=\frac 13x$ and $U_2(x)=\frac 13 x +\frac 23, \ x\in \D R , $ and $P_c$ be the associated self-similar Borel probability measure on $\D R$.   It is well-known that $P_c=\frac 12 P_c\circ U_1^{-1}+\frac 12 P_c\circ U_2^{-1}$ with support the Cantor set $C$ \cite{H}. For this probability measure Graf and Luschgy determined the optimal sets of $n$-means and the $n$th quantization errors for all $n\geq 2$  \cite{GL2}. They also proved that the quantization dimension exists, but the quantization coefficient does not exist  \cite{GL2}. Let $S$ be a Cantor dust on $\D R^2$ which is generated by a set of $4^u$ similarity mappings $\{S_j : j=1, 2, 3, \cdots, 4^u \}$, where $u\in \D N$. Let $P$ be a Borel probability measure on $\D R^2$ such that
\[P=\frac 1{4^u} \sum_{j=1}^{4^u} P\circ S_j^{-1}.\]
Then, $P$ is unique and $P$ has support the Cantor dust $S$. It is extremely difficult to calculate the optimal sets of $n$-means and the $n$th quantization errors for all $n\in \D N$ and all $u\in\D N$. Even, they were not known for $u=1$. In this paper, we investigate them for $u=1$. The Cantor dust for $u=1$ is called the \tit{standard Cantor dust}, and it is generated by the maps
$S_1(x_1, x_2)=\frac 13(x_1, x_2)$, $S_2(x_1, x_2)=\frac 13(x_1, x_2) + (\frac 23, 0)$, $S_3(x_1, x_2)=\frac 13(x_1, x_2) +(0, \frac 23)$, and $S_4(x_1, x_2)=\frac 13(x_1, x_2)+(\frac 23, \frac 23), \ (x_1, x_2) \in \D R^2 . $  The associated self-similar Borel probability measure $P$ on $\D R^2$ is given by $P=\frac 1 4P\circ S_1^{-1}+\frac 1 4P\circ S_2^{-1}+\frac 1 4P\circ S_3^{-1}+\frac 1 4P\circ S_4^{-1}$.  Then, $P$ has support $S$ \cite{H}.
For this probability measure $P$ we have determined the optimal sets of $n$-means and the $n$th quantization errors for all $n\geq 1$.  The standard (as well as more general) Cantor dust $S$ satisfies the strong separation condition (SSC), i.e., $S_1(S), S_2(S), S_3(S)$, and $S_4(S)$ are pairwise disjoint; hence, the quantization dimension of the probability measure $P$ exists and equals the Hausdorff dimension of $S$ \cite{GL2}.  In this paper, we show that the quantization coefficient for $P$ with support the standard Cantor dust does not exist.

For the (standard) Cantor distribution $P_c, $ for any $n$-points, one can easily determine whether the $n$-points form a CVT \cite{GL2}; however, for the probability distribution $P$ supported by the Cantor dust considered in this paper, for $n$-points sometimes it is quite difficult to determine whether they form a CVT.  From the construction, the Cantor dust $S$ above can be viewed as $C \times C$ and the probability distribution $P$ can be considered as $P_c\times P_c$. Hence, it might be surmised that for any $n\geq 1, $ the optimal sets of $n$-means can be deduced as products of those of $C. $  This is far from the truth; as it will be clear from the arguments in the next sections that, for instance, while $\alpha=\{\frac{1}{18}, \frac{5}{18}, \frac{5}{6}\}$ is an optimal set of 3-means for $P_c, $ the set $ \alpha \times \alpha $ is not an optimal set of 9-means for $P.$  By the same token, there is no direct connection between $n$th quantization errors in the case $C$ and $S. $

The technique we utilized can be extended to determine the optimal sets and the corresponding quantization error for many other singular continuous probability measures on $\D R^2$, such as probability measures on more general Cantor dusts generated by $4^u $ self-similar maps, $u \geq 2. $

\section{Preliminaries}

Let $I:=\{1, 2, 3, 4\}  $ be an alphabet and, for $k\geq 1 , $ call any finite sequence $\gs:=\gs_1\gs_2\cdots \gs_k , $ where $ \gs_i \in I, $ a \textit{word} $\sigma$ over $I $ with length $k. $
A word of length zero is called the \textit{empty word}, and is denoted by $\emptyset$.  By $I^*$ we denote the set of all words
over $I$ of some finite length $k$ including the empty word $\emptyset$. For any two words $\gs:=\gs_1\gs_2\cdots \gs_k$ and
$\tau:=\tau_1\tau_2\cdots \tau_\ell$ in $I^*$, by
$\gs\tau:=\gs_1\cdots \gs_k\tau_1\cdots \tau_\ell$, we mean the word obtained from the concatenation of the two words $\gs$ and $\tau$. For $\gs, \gt\in I^\ast$, $\gs$ is called {\it an extension of} $\gt$ if $\gs=\gt x$ for some word $x\in I^\ast$. The maps $S_i :\D R^2 \to \D R^2,\ 1\leq i \leq 4, $ defined above, are the generating maps of the standard Cantor dust $S . $
For $\gs:=\gs_1\gs_2 \cdots\gs_k \in I^k$, set $S_\gs:=S_{\gs_1}\circ \cdots \circ S_{\gs_k}$
and $J_\gs:=S_{\gs}([0, 1]\times [0, 1])$.  If $\emptyset , \ S_{\emptyset}=id , $ the identity map on $\D R^2$, and we will write $J:=J_{\emptyset}=S_{\emptyset}([0,1]\times [0, 1])=[0, 1]\times [0, 1]$.  Hence, $S:=\cap_{k \in \D N} \cup_{\gs \in I^k} J_\gs , $ and
the elements of the set $\{J_\gs : \gs \in I^k \}$ are the $4^k$ squares in the $k$th level in the construction of $S, $ which will be called {\it basic squares at the $k$th level.}  The squares $J_{\gs 1}$, $J_{\gs 2}$, $J_{\gs 3}$ and $J_{\gs 4}$ into which $J_\gs$ is split up at the $(k+1)$th level are called the {\it children of $J_\gs$.}  The point of intersection of the two diagonals of $J_{\gs}$ will be called the {\it center} of $J_\gs .$

A few observations on the $S$ and self-similar probability measure $P$ are in order.  Clearly, $S$ has four axes of symmetry: the lines $h \ (x_2=\frac{1}{2}),\ v \ (x_1=\frac{1}{2}),\ r \ (x_1=x_2 ) $ and $l \ (x_1+x_2=1). $  Furthermore, by the following statement the marginal distributions of $P$ are $P_c; $ since the proof is straightforward, it will not be given.

\begin{prop}  \label{prop111}
Let $P_1, P_2$ be the marginal distributions of $P$, i.e., $P_1(A)=P(A\times \D R)$ for all $A \in \F B$, and $P_2(B)=P(\D R \times B), \ \forall \ B \in \F B$, where $\F B$ is the Borel $\gs$-algebra on $\D R$. Then, $P_1=P_2=P_c$, where $P_c$ is the Cantor distribution associated with standard Cantor Set.
\end{prop}
\medskip

Since $P =\frac 1 4 P \circ S_1^{-1} + \frac 1 4 P\circ S_2^{-1} +\frac 1 4  P\circ S_3^{-1}+\frac 1 4  P\circ S_4^{-1}$, we have, by induction, that
$P=\sum_{\sigma \in I^k} \frac 1 {4^k} P\circ S_\sigma^{-1} . $  Hence, it follows that

\begin{lemma} \label{lemma1} Let $f: \D R\times \D R \to \D R^+\times \D R^+$ be Borel measurable and $k\in \D N$. Then,
\[\int f \,dP=\frac 1 {4^k} \sum_{\gs \in I^k}\int f\circ S_\gs \,dP.\]
\end{lemma}
\medskip

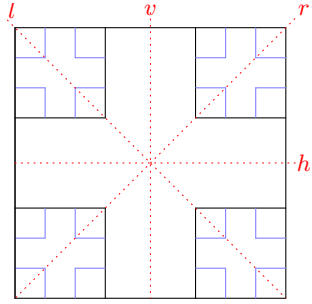
\begin{figure}
\noindent \begin{tikzpicture}[line cap=round,line join=round,>=triangle 45,x=0.4 cm,y=0.4 cm]
\clip(-0.3,-0.3) rectangle (9.8,9.8);
\draw (0.,9.)-- (9.,9.);
\draw (9.,9.)-- (9.,0.);
\draw (0.,3.)-- (3.,3.);
\draw (3.,3.)-- (3.,0.);
\draw (0.,6.)-- (3.,6.);
\draw (3.,9.)-- (3.,6.);
\draw (6.,9.)-- (6.,6.);
\draw (6.,6.)-- (9.,6.);
\draw (6.,3.)-- (6.,0.);
\draw (6.,3.)-- (9.,3.);
\draw (0.,0.)-- (0.,9.);
\draw (0.,0.)-- (9.,0.);
\draw [color=xdxdff]  (1.,3.)-- (1.,2.);
\draw [color=xdxdff]  (1.,2.)-- (0.,2.);
\draw [color=xdxdff]  (1.,1.)-- (0.,1.);
\draw [color=xdxdff]  (1.,1.)-- (1.,0.);
\draw [color=xdxdff]  (2.,3.)-- (2.,2.);
\draw [color=xdxdff]  (2.,2.)-- (3.,2.);
\draw [color=xdxdff]  (2.,1.)-- (3.,1.);
\draw [color=xdxdff]  (2.,1.)-- (2.,0.);
\draw [color=xdxdff]  (7.,3.)-- (7.,2.);
\draw [color=xdxdff]  (6.,2.)-- (7.,2.);
\draw [color=xdxdff]  (8.,3.)-- (8.,2.);
\draw [color=xdxdff]  (8.,2.)-- (9.,2.);
\draw [color=xdxdff]  (6.,1.)-- (7.,1.);
\draw [color=xdxdff]  (7.,1.)-- (7.,0.);
\draw [color=xdxdff]  (8.,1.)-- (9.,1.);
\draw [color=xdxdff]  (8.,1.)-- (8.,0.);
\draw [color=xdxdff]  (1.,9.)-- (1.,8.);
\draw [color=xdxdff]  (0.,8.)-- (1.,8.);
\draw [color=xdxdff]  (1.,7.)-- (0.,7.);
\draw [color=xdxdff]  (1.,7.)-- (1.,6.);
\draw [color=xdxdff]  (2.,9.)-- (2.,8.);
\draw [color=xdxdff]  (3.,8.)-- (2.,8.);
\draw [color=xdxdff]  (3.,7.)-- (2.,7.);
\draw [color=xdxdff]  (2.,6.)-- (2.,7.);
\draw [color=xdxdff]  (7.,9.)-- (7.,8.);
\draw [color=xdxdff]  (7.,8.)-- (6.,8.);
\draw [color=xdxdff]  (8.,9.)-- (8.,8.);
\draw [color=xdxdff]  (9.,8.)-- (8.,8.);
\draw [color=xdxdff]  (8.,7.)-- (9.,7.);
\draw [color=xdxdff]  (8.,7.)-- (8.,6.);
\draw [color=xdxdff]  (7.,7.)-- (6.,7.);
\draw [color=xdxdff]  (7.,7.)-- (7.,6.);
\draw (0.,0.)-- (0.,9.);
\draw [dotted,color=ffqqqq] (9.3,9.3)-- (0.,0.);
\draw [dotted,color=ffqqqq] (-0.4,9.4)-- (9.,0.);
\draw [dotted,color=ffqqqq] (0.,4.5)-- (9.4,4.5);
\draw [dotted,color=ffqqqq] (4.5,9.5)-- (4.5,0);
\begin{scriptsize}
\draw[color=ffqqqq] (-0.1, 9.6) node {$l$};
\draw[color=ffqqqq] (9.6,9.6) node {$r$};
\draw[color=ffqqqq] (9.6,4.5) node {$h$};
\draw[color=ffqqqq] (4.5,9.6) node {$v$};
\end{scriptsize}
\end{tikzpicture}
\caption{Symmetry axes of the Cantor dust.} \label{Fig1}
\end{figure}

\begin{remark} \label{rem1}  With respect to the horizontal and vertical lines $h$ and $v$, and the diagonals $r$ and $l$, the Cantor dust $S$ has \tit{maximum symmetry}, i.e., with respect to the four lines, the Cantor dust is geometrically symmetric as well as symmetric with respect to the probability distribution $P$ (see Figure~\ref{Fig1}). By the symmetry with respect to $P$, it is meant that if the two basic squares of similar geometrical shape lie in the opposite sides and are equidistant with respect to any of the lines $h$, $v$, $r$, and $l$, then they have the same probability. Notice that the Cantor dust has four lines of maximum symmetry. Due to this fact, we conjecture that among all the pairs of points whose Voronoi regions have the boundaries as oblique lines passing through the point $(\frac 12, \frac 12)$, the pair of points whose Voronoi regions have the boundary as any of the two diagonals will give the smallest distortion error.
\end{remark}
\medskip

\begin{notation}  As usual, let $X=(X_1, X_2)$ be a random vector with distribution $P, \ E(X)$ and $V(X)$ will denote the expected vector and the expected squared Euclidean distance of $X , $ respectively.
For words $\gb, \gg, \cdots, \gd \in I^\ast $, by $a(\gb, \gg, \cdots, \gd)$ we will mean the conditional expectation of the random variable $X$ given $J_\gb\uu J_\gg \uu\cdots \uu J_\gd,$ i.e.,
\begin{equation} \label{eq000}a(\gb, \gg, \cdots, \gd)=E(X|X\in J_\gb \uu J_\gg \uu \cdots \uu J_\gd)=\frac{1}{P(J_\gb\uu \cdots \uu J_\gd)}\int_{J_\gb\uu \cdots \uu J_\gd} x dP,
\end{equation}
where $\int x dP =\int(x_1, x_2) dP$.
\end{notation}

\medskip

\begin{lemma} \label{lemma333}
Let $X:=(X_1, X_2)$ be a random vector with distribution $P.$  Then, $E(X)=(E(X_1), \, E(X_2))=(\frac 12, \frac 12) \te{ and } V:=V(X)=E\|X-(\frac 1 2, \frac 1 2)\|^2=\frac 1 4.$
\end{lemma}
\begin{proof} Since $P_1$ and $P_2$ are the marginal distributions of $X:=(X_1, X_2)$, the random variables $X_1$ and $X_2$ have distributions $P_1$ and $P_2$, respectively.  Since, by Proposition~\ref{prop111}, $P_1=P_2=P_c$, both $X_1$ and $X_2$ are $P_c$-distributed random variables. Thus, by Lemma~3.4  \cite{GL2}, we obtain that $E(X_1)=E(X_2)=\frac 12$, and $V(X_1)=V(X_2)=\frac 18 . $
Consequently, it follows that $E\|X-(\frac 12, \frac 12)\|^2=E(X_1-\frac 12)^2 +E(X_2-\frac 12)^2=V(X_1)+V(X_2)=\frac 14. $
\end{proof}
\medskip

\begin{remark} \label{rem11}  For any $(a, b) \in \D R^2$, \begin{align*} E\|X-(a, b)\|^2 &=\iint_{\D R^2} [(x_1-a)^2+(x_2-b)^2]\, dP(x_1, x_2)\\
&=\int_{0}^1 (x_1-a)^2 \,dP_1(x_1)+\int_0^1(x_2-b)^2 \,dP_2(x_2)=E(X_1-a)^2 +E(X_2-b)^2\\
&=V(X_1)+V(X_2)+(a-\frac 1 2)^2+(b-\frac 12)^2=V+\|(a, b)-(\frac 12, \frac 12)\|^2.
\end{align*}
Hence, by Lemma \ref{lemma1} , for any $\gs \in I^k$, $k\geq 1$, we have
\[\int_{J_\gs}\|x-(a, b)\|^2 dP= \frac{1}{4^k} \int\|x -(a, b)\|^2 dP\circ S_\gs^{-1},\]
which implies
\begin{equation} \label{eq1}
\int_{J_\gs}\|x-(a, b)\|^2 dP=\frac{1}{4^k}\Big(\frac{1}{9^k}V+\|S_\gs(\frac 12, \frac 12)-(a, b)\|^2\Big).
\end{equation}

\end{remark}

\bigskip

\section{Optimal sets of $n$-means and the quantization errors for all $n\geq 2$}

In this section, we determine the optimal sets of $n$-means, $n\geq 1. $
First, we will determine optimal sets of $n$-means for $n=1,2,3 , $ which will form the base in determining optimal sets of $n$-means for $ n\geq 4.$
To determine the quantization error, we will frequently use the relation \eqref{eq1}.  Below, by an {\it oblique line} will mean any line which is neither horizontal nor vertical.
\begin{figure}
\begin{tikzpicture}[line cap=round,line join=round,>=triangle 45,x=0.6cm,y=0.6cm]
\clip(-0.3,-1.18) rectangle (6.6,6.6);
\draw (0.,6.)-- (6.,6.);
\draw (6.,6.)-- (6.,0.);
\draw (0.,6.)-- (0.,0.);
\draw (2.,6.)-- (2.,4.);
\draw (0.,4.)-- (2.,4.);
\draw (4.,6.)-- (4.,4.);
\draw (4.,2.)-- (6.,2.);
\draw (4.,2.)-- (4.,0.);
\draw (0.,2.)-- (2.,2.);
\draw (2.,2.)-- (2.,0.);
\draw (4.,4.)-- (6.,4.);
\draw [dotted,color=ffqqqq] (3.,5.)-- (1.,1.);
\draw [dotted,color=ffqqqq] (1.,1.)-- (5.,1.);
\draw [dotted,color=ffqqqq] (5.,1.)-- (3.,5.);
\draw [dotted,color=ffqqqq] (0.,4.)-- (3.,2.5);
\draw [dotted,color=ffqqqq] (6.,4.)-- (3.,2.5);
\draw (0.,0.)-- (6.,0.);
\draw [dotted,color=ffqqqq] (3.,2.5)-- (3.,0.);
\begin{scriptsize}
\draw [fill=xdxdff] (0.,6.) circle (1.5pt);
\draw[color=xdxdff] (0.08192582624720256,6.362937364327278) node {$C$};
\draw [fill=qqqqff] (6.,6.) circle (1.5pt);
\draw[color=qqqqff] (6.130285699666685,6.307192112037698) node {$B$};
\draw [fill=xdxdff] (0.,4.) circle (1.5pt);
\draw[color=xdxdff] (0.19341633082636353,4.383980908047175) node {$U$};
\draw [fill=xdxdff] (6.,4.) circle (1.5pt);
\draw[color=xdxdff] (6.186030951956266,4.383980908047175) node {$T$};
\draw [fill=qqqqff] (1.,1.) circle (1.5pt);
\draw[color=qqqqff] (0.8523593583013294,0.7026668830796629) node {$P$};
\draw [fill=qqqqff] (5.,1.) circle (1.5pt);
\draw[color=qqqqff] (5.210489036888608, 0.7026668830796629) node {$Q$};
\draw [fill=qqqqff] (3.,5.) circle (1.5pt);
\draw[color=qqqqff] (3.20365995446371,5.387395449259622) node {$R$};
\draw [fill=qqqqff] (3.,2.5) circle (1.5pt);
\draw[color=qqqqff] (3.20365995446371,2.8788590962285068) node {$W$};
\draw [fill=qqqqff] (3.,0.) circle (1.5pt);
\draw[color=qqqqff] (3.1479147021741296,0.28670486476302004) node {$S$};
\draw [fill=qqqqff] (0.,0.) circle (1.5pt);
\draw[color=qqqqff] (0.20767107853678304,0.28670486476302004) node {$O$};
\draw [fill=qqqqff] (6.,0.) circle (1.5pt);
\draw[color=qqqqff] (6.186030951956266,0.3145774909078102) node {$A$};
\draw[] (3.1479147021741296,-0.910486476302004) node { $(a)$ };
\end{scriptsize}
\end{tikzpicture}  \qquad
\begin{tikzpicture}[line cap=round,line join=round,>=triangle 45,x=0.6 cm,y=0.6cm]
\clip(-0.4,-1.18) rectangle (6.6,6.6);
\draw (0.,6.)-- (6.,6.);
\draw (6.,6.)-- (6.,0.);
\draw (0.,6.)-- (0.,0.);
\draw (2.,6.)-- (2.,4.);
\draw (0.,4.)-- (2.,4.);
\draw (4.,6.)-- (4.,4.);
\draw (4.,2.)-- (6.,2.);
\draw (4.,2.)-- (4.,0.);
\draw (0.,2.)-- (2.,2.);
\draw (2.,2.)-- (2.,0.);
\draw (4.,4.)-- (6.,4.);
\draw (0.,0.)-- (6.,0.);
\draw [dotted,color=ffqqqq] (5.,5.)-- (0.8666666666666667,3.8);
\draw [dotted,color=ffqqqq] (5.,5.)-- (3.8,0.8666666666666667);
\draw [dotted,color=ffqqqq] (3.8,0.8666666666666667)-- (0.8666666666666667,3.8);
\draw [dotted,color=ffqqqq] (2.4570987471122963,6.001453400434421)-- (3.2439764417261774,3.247381469285843);
\draw [dotted,color=ffqqqq] (3.2439764417261774,3.247381469285843)-- (0.,0.);
\draw [dotted,color=ffqqqq] (3.2439764417261774,3.247381469285843)-- (6.0177203152401075,2.4801757170373104);
\begin{scriptsize}
\draw [fill=qqqqff] (0.8666666666666667,3.8) circle (1.5pt);
\draw[color=qqqqff] (0.6151951115024872,3.568781663650882) node {$P_1$};
\draw [fill=qqqqff] (3.8,0.8666666666666667) circle (1.5pt);
\draw[color=qqqqff] (3.678647680230883,0.5100625160755971) node {$Q_1$};
\draw [fill=qqqqff] (5.,5.) circle (1.5pt);
\draw[color=qqqqff] (5.171098931662666,5.375433178541985) node {$R_1$};
\draw [fill=qqqqff] (3.2439764417261774,3.247381469285843) circle (1.5pt);
\draw[color=qqqqff] (3.6000976143660526,3.2807647554798365) node {$W$};
\draw [fill=qqqqff] (2.4570987471122963,6.001453400434421) circle (1.5pt);
\draw[color=qqqqff] (2.6313134686998074,6.300400679496505) node {$S$};
\draw [fill=qqqqff] (6.0177203152401075,2.4801757170373104) circle (1.5pt);
\draw[color=qqqqff] (6.192249787905464,2.83564771557913) node {$T$};
\draw [fill=uuuuuu] (0.,0.) circle (1.5pt);
\draw[color=uuuuuu] (0.20767107853678304,0.28670486476302004) node {$O$};
\draw [fill=uuuuuu] (1.99790290879475,2.) circle (1.5pt);
\draw[color=uuuuuu] (2.1861964287991005,2.3043473203901464) node {$B$};
\draw [fill=uuuuuu] (0.,2.) circle (1.5pt);
\draw[color=uuuuuu] (0.196261426890057,2.3043473203901464) node {$C$};
\draw [fill=uuuuuu] (2.,0.) circle (1.5pt);
\draw[color=uuuuuu] (2.1861964287991005,0.27441231848110545) node {$A$};
\draw[] (3.1479147021741296,-0.910486476302004) node { $(b)$ };
\end{scriptsize}
\end{tikzpicture}
\caption{$(a)$ CVT with three-means $P$, $Q$ and $R$ with one on the vertical line of the symmetry; $(b)$ CVT with three-means $P_1$, $Q_1$ and $R_1$ with one on a diagonal of the square.} \label{Fig2}
\end{figure}
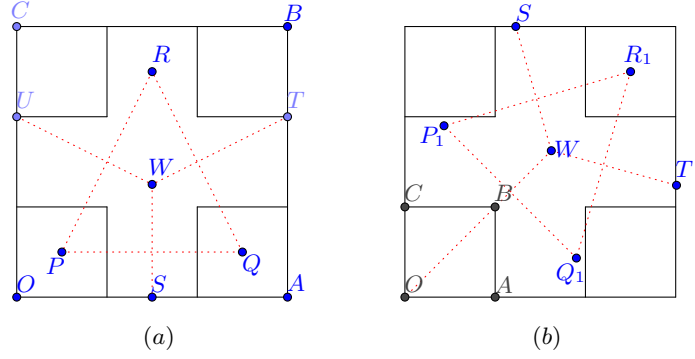

First, from Lemma~\ref{lemma333}  it follows that the optimal set of one-mean is $E(X)=\{(\frac{1}{2}, \frac{1}{2})\}$ and the corresponding quantization error is $V (X)=\frac{1}{4}$ of the random variable $X . $
For any word $\sigma \in I^k$, $k\geq 1$, since $a(\gs)=E(X | X \in J_\sigma)$, using Lemma~\ref{lemma1}, we have
\begin{align*}
&a(\gs)=\frac{1}{P(J_\sigma)} \int_{J_\sigma} x \,dP(x)=\int_{J_\sigma} x\, dP\circ S_\sigma^{-1}(x)=\int S_\sigma(x)\, dP(x)=E(S_\sigma(X)).
\end{align*}
Since $S_i$ are similarity mappings, it follows that $E(S_j(X))=S_j(E(X))$ for $j=1, 2, 3, 4; $ and so, by induction,
$$a(\gs)=E(S_\sigma(X))=S_\sigma(E(X))=S_\sigma(\frac 12, \frac 12),\
\text{for}\ \sigma\in I^k,\ k\geq 1 . $$
From this observation we surmise that the sets $\set{(\frac 1 6, \frac 12), (\frac 56, \frac 12)}$ and $\set{(\frac 12, \frac 16), (\frac 12, \frac 56)}$ form two different optimal sets of two-means.  In order to prove this fact, first we will show that an optimal set of two-means cannot lie on an oblique line.

\begin{lemma} \label{lemma301} The points in an optimal set of two-means cannot lie on a oblique line of $S.$
\end{lemma}

\begin{proof}  The boundary of the Voronoi regions generated by the points in $\gb:=\set{(\frac 1 6, \frac 12), (\frac 56, \frac 12)}$ is the line $ v ; $ i.e., $J_1\uu J_3\sci M((\frac 16, \frac 12)|\gb)$ and $J_2\uu J_4\sci M((\frac 56, \frac 12)|\gb)$.
Let $V_{2,1}$ be the distortion error due to the set $\gb$. Then,
\begin{align} \label{eq4}
V_{2, 1}=\mathop{\int}\limits_{J_1\uu J_3} \|x-(\frac 16, \frac 12)\|^2 dP+\mathop{\int}\limits_{J_3\uu J_4} \|x-(\frac 56, \frac 12)\|^2  dP=\frac 5{36}.\notag \end{align}
Among all the oblique lines through the point $(\frac 12, \frac 12)$, as mentioned in Remark~\ref{rem1}, the pair of points whose Voronoi regions have the boundary as any of the two diagonals will give the smallest distortion error.

We also know that the points which give the smallest distortion error are the centroids of their own Voronoi regions. Let $(a_1, b_1)$ and $(a_2, b_2)$ be the centroids of the left half and the right half of the Cantor dust with respect to the diagonal passing through the origin.
Thus, using \eqref{eq000}, we have
\begin{align*}
&(a_1, a_2)=E\left(X| X\in J_2\uu(J_{12}\uu J_{42})\uu (J_{112}\uu J_{142}\uu J_{412}\uu J_{442})\uu \cdots\right)\\
&=2\Big(\frac 14S_2(\frac 1 2, \frac 12)+ \frac{1}{4^2} \Big(S_{12}(\frac 1 2, \frac 12)+S_{42}(\frac 1 2, \frac 12) \Big)+\frac{1}{4^3} \Big(S_{112}(\frac 1 2, \frac 12)+S_{142}(\frac 1 2, \frac 12) \notag\\
&\qquad \qquad +S_{412}(\frac 1 2, \frac 12)+S_{442}(\frac 1 2, \frac 12) \Big)+\cdots\Big)\notag\\
&=2\Big((\frac{5}{24}, \frac{1}{24})+(\frac{11}{144},\frac{7}{144})+(\frac{29}{864}, \frac{25}{864})+(\frac{83}{5184}, \frac{79}{5184})+\cdots\Big)\notag\\
&=2\Big(\frac{5}{24}+\frac{11}{144}+\frac{29}{864}+\frac{83}{5184}+\cdots, \, \frac{1}{24}+\frac{7}{144}+\frac{25}{864}+\frac{79}{5184}+\cdots\Big)=2\Big(\frac{7}{20}, \frac{3}{20}\Big),\notag
\end{align*}
i.e, \begin{equation} \label{eq67} (a_1, a_2)=\Big(\frac{7}{10}, \frac{3}{10}\Big). \end{equation} Similarly, one can show that
\begin{equation} \label{eq68} (a_2, b_2)=E\left(X| X\in J_3\uu(J_{13}\uu J_{43})\uu (J_{113}\uu J_{143}\uu J_{413}\uu J_{443})\uu \cdots\right)=\Big(\frac 3{10}, \frac 7{10}\Big).
\end{equation}
Hence, if  $V_{2,2}$ is the distortion error due to the points $(\frac{7}{10}, \frac{3}{10})$ and $(\frac{3}{10}, \frac{7}{10})$, then we have
\begin{align*} V_{2,2}&=2\Big(\te{distortion error due to } (\frac{7}{10}, \frac{3}{10}))>2 \Big(\int_{J_2}\|(x_1, x_2)-(\frac{7}{10}, \frac{3}{10})\|^2 dP\\
& \qquad +\int_{J_{12}\uu J_{42}}\|(x_1, x_2)-(\frac{7}{10}, \frac{3}{10})\|^2 dP+\int_{J_{112}\uu J_{412}\uu J_{142}\uu J_{442}}\|(x_1, x_2)-(\frac{7}{10}, \frac{3}{10})\|^2 dP\\
& \qquad  +\int_{J_{1112}\uu J_{1412}\uu J_{1142}\uu J_{1442}\uu J_{4112}\uu J_{4412}\uu J_{4142}\uu J_{4442}}\|(x_1, x_2)-(\frac{7}{10}, \frac{3}{10})\|^2 dP\Big)\\
&=2 (0.0747492)=0.1494984> \frac{5}{36} =V_{2,1}.
\end{align*}
Therefore, the points in an optimal set of two-means cannot lie on a oblique line of the Cantor dust $S. $
\end{proof}
\begin{figure}
\begin{tikzpicture}[line cap=round,line join=round,>=triangle 45,x=0.4cm,y=0.4cm]
\clip(-0.2,-0.2) rectangle (6.9,6.9);
\draw (0.,6.6)-- (6.6,6.6);
\draw (6.6,6.6)-- (6.6,0.);
\draw (0.,2.2)-- (2.2,2.2);
\draw (2.2,2.2)-- (2.2,0.);
\draw (0.,4.4)-- (2.2,4.4);
\draw (2.2,6.6)-- (2.2,4.4);
\draw (4.4,6.6)-- (4.4,4.4);
\draw (4.4,4.4)-- (6.6,4.4);
\draw (4.4,2.2)-- (4.4,0.);
\draw (4.4,2.2)-- (6.6,2.2);
\draw (0.,0.)-- (6.6,0.014820793092255258);
\draw (0.,2.2)-- (0.,0.);
\draw (0.,6.6)-- (0.,0.);
\begin{scriptsize}
\draw [fill=ffqqqq] (3.3,1.1) circle (1.5pt);
\draw [fill=ffqqqq] (3.3,5.3) circle (1.5pt);
\end{scriptsize}
\end{tikzpicture}\qquad
\begin{tikzpicture}[line cap=round,line join=round,>=triangle 45,x=0.4cm,y=0.4cm]
\clip(-0.2,-0.2) rectangle (6.9,6.9);
\draw (0.,6.6)-- (6.6,6.6);
\draw (6.6,6.6)-- (6.6,0.);
\draw (0.,2.2)-- (2.2,2.2);
\draw (2.2,2.2)-- (2.2,0.);
\draw (0.,4.4)-- (2.2,4.4);
\draw (2.2,6.6)-- (2.2,4.4);
\draw (4.4,6.6)-- (4.4,4.4);
\draw (4.4,4.4)-- (6.6,4.4);
\draw (4.4,2.2)-- (4.4,0.);
\draw (4.4,2.2)-- (6.6,2.2);
\draw (0.,0.)-- (6.6,0.014820793092255258);
\draw (0.,2.2)-- (0.,0.);
\draw (0.,6.6)-- (0.,0.);
\begin{scriptsize}
\draw [fill=ffqqqq] (1.1,1.1) circle (1.5pt);
\draw [fill=ffqqqq] (5.5,1.1) circle (1.5pt);
\draw [fill=ffqqqq] (3.3,5.3) circle (1.5pt);
\end{scriptsize}
\end{tikzpicture}\qquad
\begin{tikzpicture}[line cap=round,line join=round,>=triangle 45,x=0.4cm,y=0.4cm]
\clip(-0.2,-0.2) rectangle (6.9,6.9);
\draw (0.,6.6)-- (6.6,6.6);
\draw (6.6,6.6)-- (6.6,0.);
\draw (0.,2.2)-- (2.2,2.2);
\draw (2.2,2.2)-- (2.2,0.);
\draw (0.,4.4)-- (2.2,4.4);
\draw (2.2,6.6)-- (2.2,4.4);
\draw (4.4,6.6)-- (4.4,4.4);
\draw (4.4,4.4)-- (6.6,4.4);
\draw (4.4,2.2)-- (4.4,0.);
\draw (4.4,2.2)-- (6.6,2.2);
\draw (0.,0.)-- (6.6,0.014820793092255258);
\draw (0.,2.2)-- (0.,0.);
\draw (0.,6.6)-- (0.,0.);
\begin{scriptsize}
\draw [fill=ffqqqq] (1.1,1.1) circle (1.5pt);
\draw [fill=ffqqqq] (5.5,1.1) circle (1.5pt);
\draw [fill=ffqqqq] (1.1,5.5) circle (1.5pt);
\draw [fill=ffqqqq] (5.5,5.5) circle (1.5pt);
\end{scriptsize}
\end{tikzpicture}\qquad
\noindent \begin{tikzpicture}[line cap=round,line join=round,>=triangle 45,x=0.4 cm,y=0.4 cm]
\clip(-0.2,-0.2) rectangle (6.9,6.9);
\draw (0.,6.6)-- (6.6,6.6);
\draw (6.6,6.6)-- (6.6,0.);
\draw (0.,2.2)-- (2.2,2.2);
\draw (2.2,2.2)-- (2.2,0.);
\draw (0.,4.4)-- (2.2,4.4);
\draw (2.2,6.6)-- (2.2,4.4);
\draw (4.4,6.6)-- (4.4,4.4);
\draw (4.4,4.4)-- (6.6,4.4);
\draw (4.4,2.2)-- (4.4,0.);
\draw (4.4,2.2)-- (6.6,2.2);
\draw (0.015163841011398968,6.6)-- (0.,0.);
\draw (0.,0.)-- (6.6,0.014820793092255258);
\draw (0.7,2.2)-- (0.7,1.4);
\draw (0.7,1.4)-- (0.,1.4);
\draw (0.7,0.7)-- (0.,0.7);
\draw (0.7,0.7)-- (0.7,0.);
\draw (1.4,2.2)-- (1.4,1.4);
\draw (1.4,1.4)-- (2.2,1.4);
\draw (1.4,0.7)-- (2.2,0.7);
\draw (1.4,0.7)-- (1.4,0.);
\begin{scriptsize}
\draw [fill=ffqqqq] (1.1,0.35) circle (1.5pt);
\draw [fill=ffqqqq] (1.1,1.8) circle (1.5pt);
\draw [fill=ffqqqq] (5.5,1.1) circle (1.5pt);
\draw [fill=ffqqqq] (1.1,5.5) circle (1.5pt);
\draw [fill=ffqqqq] (5.5,5.5) circle (1.5pt);
\end{scriptsize}
\end{tikzpicture}
\caption{Optimal configuration of $n$ points for $2\leq n\leq 5 . $}

\end{figure}
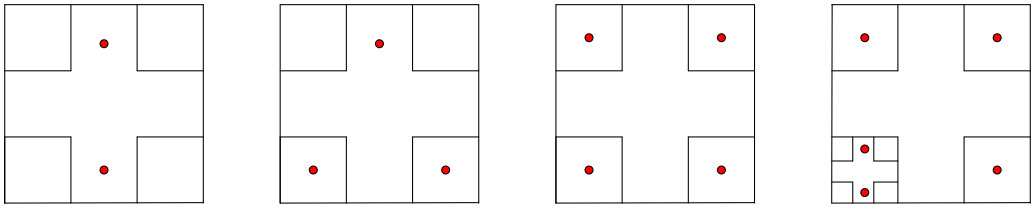

\vspace{.5 in}

\begin{figure}
\noindent \begin{tikzpicture}[line cap=round,line join=round,>=triangle 45,x=0.4 cm,y=0.4 cm]
\clip(-0.2,-0.2) rectangle (9.2,9.2);
\draw (0.,9.)-- (9.,9.);
\draw (9.,9.)-- (9.,0.);
\draw (0.,3.)-- (3.,3.);
\draw (3.,3.)-- (3.,0.);
\draw (0.,6.)-- (3.,6.);
\draw (3.,9.)-- (3.,6.);
\draw (6.,9.)-- (6.,6.);
\draw (6.,6.)-- (9.,6.);
\draw (6.,3.)-- (6.,0.);
\draw (6.,3.)-- (9.,3.);
\draw (0.015163841011398968,9.)-- (0.,0.);
\draw (0.,0.)-- (9.,0.014820793092255258);
\draw (1.,3.)-- (1.,2.);
\draw (1.,2.)-- (0.,2.);
\draw (1.,1.)-- (0.,1.);
\draw (1.,1.)-- (1.,0.);
\draw (2.,3.)-- (2.,2.);
\draw (2.,2.)-- (3.,2.);
\draw (2.,1.)-- (3.,1.);
\draw (2.,1.)-- (2.,0.);
\draw (7.,3.)-- (7.,2.);
\draw (6.,2.)-- (7.,2.);
\draw (8.,3.)-- (8.,2.);
\draw (8.,2.)-- (9.,2.);
\draw (6.,1.)-- (7.,1.);
\draw (7.,1.)-- (7.,0.);
\draw (8.,1.)-- (9.,1.);
\draw (8.,1.)-- (8.,0.);
\draw (1.,9.)-- (1.,8.);
\draw (0.,8.)-- (1.,8.);
\draw (1.,7.)-- (0.,7.);
\draw (1.,7.)-- (1.,6.);
\draw (2.,9.)-- (2.,8.);
\draw (3.,8.)-- (2.,8.);
\draw (3.,7.)-- (2.,7.);
\draw (2.,6.)-- (2.,7.);
\draw (7.,9.)-- (7.,8.);
\draw (7.,8.)-- (6.,8.);
\draw (8.,9.)-- (8.,8.);
\draw (9.,8.)-- (8.,8.);
\draw (8.,7.)-- (9.,7.);
\draw (8.,7.)-- (8.,6.);
\draw (7.,7.)-- (6.,7.);
\draw (7.,7.)-- (7.,6.);
\begin{scriptsize}
\draw [fill=ffqqqq] (6.5,0.5) circle (1.5pt);
\draw [fill=ffqqqq] (0.5,0.5) circle (1.5pt);
\draw [fill=ffqqqq] (2.5,0.5) circle (1.5pt);
\draw [fill=ffqqqq] (8.5,0.5) circle (1.5pt);
\draw [fill=ffqqqq] (0.5,6.5) circle (1.5pt);
\draw [fill=ffqqqq] (2.5,6.5) circle (1.5pt);
\draw [fill=ffqqqq] (6.5,6.5) circle (1.5pt);
\draw [fill=ffqqqq] (8.5,6.5) circle (1.5pt);
\draw [fill=ffqqqq] (0.5,2.5) circle (1.5pt);
\draw [fill=ffqqqq] (2.5,2.5) circle (1.5pt);
\draw [fill=ffqqqq] (6.5,2.5) circle (1.5pt);
\draw [fill=ffqqqq] (8.5,2.5) circle (1.5pt);
\draw [fill=ffqqqq] (0.5,8.5) circle (1.5pt);
\draw [fill=ffqqqq] (2.5,8.5) circle (1.5pt);
\draw [fill=ffqqqq] (6.5,8.5) circle (1.5pt);
\draw [fill=ffqqqq] (8.5,8.5) circle (1.5pt);
\end{scriptsize}
\end{tikzpicture}\qquad
\begin{tikzpicture}[line cap=round,line join=round,>=triangle 45,x=0.4 cm,y=0.4cm]
\clip(-0.2,-0.2) rectangle (9.2,9.2);
\draw (0.,9.)-- (9.,9.);
\draw (9.,9.)-- (9.,0.);
\draw (0.,3.)-- (3.,3.);
\draw (3.,3.)-- (3.,0.);
\draw (0.,6.)-- (3.,6.);
\draw (3.,9.)-- (3.,6.);
\draw (6.,9.)-- (6.,6.);
\draw (6.,6.)-- (9.,6.);
\draw (6.,3.)-- (6.,0.);
\draw (6.,3.)-- (9.,3.);
\draw (0.015163841011398968,9.)-- (0.,0.);
\draw (0.,0.)-- (9.,0.014820793092255258);
\draw (1.,3.)-- (1.,2.);
\draw (1.,2.)-- (0.,2.);
\draw (1.,1.)-- (0.,1.);
\draw (1.,1.)-- (1.,0.);
\draw (2.,3.)-- (2.,2.);
\draw (2.,2.)-- (3.,2.);
\draw (2.,1.)-- (3.,1.);
\draw (2.,1.)-- (2.,0.);
\draw (7.,3.)-- (7.,2.);
\draw (6.,2.)-- (7.,2.);
\draw (8.,3.)-- (8.,2.);
\draw (8.,2.)-- (9.,2.);
\draw (6.,1.)-- (7.,1.);
\draw (7.,1.)-- (7.,0.);
\draw (8.,1.)-- (9.,1.);
\draw (8.,1.)-- (8.,0.);
\draw (1.,9.)-- (1.,8.);
\draw (0.,8.)-- (1.,8.);
\draw (1.,7.)-- (0.,7.);
\draw (1.,7.)-- (1.,6.);
\draw (2.,9.)-- (2.,8.);
\draw (3.,8.)-- (2.,8.);
\draw (3.,7.)-- (2.,7.);
\draw (2.,6.)-- (2.,7.);
\draw (7.,9.)-- (7.,8.);
\draw (7.,8.)-- (6.,8.);
\draw (8.,9.)-- (8.,8.);
\draw (9.,8.)-- (8.,8.);
\draw (8.,7.)-- (9.,7.);
\draw (8.,7.)-- (8.,6.);
\draw (7.,7.)-- (6.,7.);
\draw (7.,7.)-- (7.,6.);
\draw (0.3404138429913868,0.3535497610962645)-- (5.960646256589087E-4,0.3537745896239297);
\draw (0.3404138429913868,0.3535497610962645)-- (0.3423142276326922,0.0038789870960769203);
\draw (0.3423142276326922,0.0038789870960769203)-- (0.3404138429913868,0.3535497610962645);
\draw (0.6735359341589903,0.35347549349725715)-- (0.6719546768124955,0.00243636257541265);
\draw (0.6735359341589903,0.35347549349725715)-- (1.,0.35189423615076243);
\draw (0.3423142276326922,0.0038789870960769203)-- (0.3404138429913868,0.3535497610962645);
\draw (0.34611499691530295,0.6576113037051228)-- (0.0011111790167362308,0.6595038251263912);
\draw (0.34611499691530295,0.6576113037051228)-- (0.34611499691530295,1.001580923781394);
\draw (0.6729811552198289,0.661412072987734)-- (0.6748815398611343,1.);
\draw (0.6729811552198289,0.661412072987734)-- (1.,0.6576113037051232);
\begin{scriptsize}
\draw [fill=ffqqqq] (6.5,0.5) circle (1.5pt);
\draw [fill=ffqqqq] (2.5,0.5) circle (1.5pt);
\draw [fill=ffqqqq] (8.5,0.5) circle (1.5pt);
\draw [fill=ffqqqq] (0.5,6.5) circle (1.5pt);
\draw [fill=ffqqqq] (2.5,6.5) circle (1.5pt);
\draw [fill=ffqqqq] (6.5,6.5) circle (1.5pt);
\draw [fill=ffqqqq] (8.5,6.5) circle (1.5pt);
\draw [fill=ffqqqq] (0.5,2.5) circle (1.5pt);
\draw [fill=ffqqqq] (2.5,2.5) circle (1.5pt);
\draw [fill=ffqqqq] (6.5,2.5) circle (1.5pt);
\draw [fill=ffqqqq] (8.5,2.5) circle (1.5pt);
\draw [fill=ffqqqq] (0.5,8.5) circle (1.5pt);
\draw [fill=ffqqqq] (2.5,8.5) circle (1.5pt);
\draw [fill=ffqqqq] (6.5,8.5) circle (1.5pt);
\draw [fill=ffqqqq] (8.5,8.5) circle (1.5pt);
\draw [fill=ffqqqq] (0.5,0.162) circle (1.5pt);
\draw [fill=ffqqqq] (0.5,0.83) circle (1.5pt);
\end{scriptsize}
\end{tikzpicture}\qquad
\begin{tikzpicture}[line cap=round,line join=round,>=triangle 45,x=0.4 cm,y=0.4 cm]
\clip(-0.2,-0.2) rectangle (9.2,9.2);
\draw (0.,9.)-- (9.,9.);
\draw (9.,9.)-- (9.,0.);
\draw (0.,3.)-- (3.,3.);
\draw (3.,3.)-- (3.,0.);
\draw (0.,6.)-- (3.,6.);
\draw (3.,9.)-- (3.,6.);
\draw (6.,9.)-- (6.,6.);
\draw (6.,6.)-- (9.,6.);
\draw (6.,3.)-- (6.,0.);
\draw (6.,3.)-- (9.,3.);
\draw (0.,0.)-- (9.,0.014820793092255258);
\draw (1.,3.)-- (1.,2.);
\draw (1.,2.)-- (0.,2.);
\draw (1.,1.)-- (0.,1.);
\draw (1.,1.)-- (1.,0.);
\draw (2.,3.)-- (2.,2.);
\draw (2.,2.)-- (3.,2.);
\draw (2.,1.)-- (3.,1.);
\draw (2.,1.)-- (2.,0.);
\draw (7.,3.)-- (7.,2.);
\draw (6.,2.)-- (7.,2.);
\draw (8.,3.)-- (8.,2.);
\draw (8.,2.)-- (9.,2.);
\draw (6.,1.)-- (7.,1.);
\draw (7.,1.)-- (7.,0.);
\draw (8.,1.)-- (9.,1.);
\draw (8.,1.)-- (8.,0.);
\draw (1.,9.)-- (1.,8.);
\draw (0.,8.)-- (1.,8.);
\draw (1.,7.)-- (0.,7.);
\draw (1.,7.)-- (1.,6.);
\draw (2.,9.)-- (2.,8.);
\draw (3.,8.)-- (2.,8.);
\draw (3.,7.)-- (2.,7.);
\draw (2.,6.)-- (2.,7.);
\draw (7.,9.)-- (7.,8.);
\draw (7.,8.)-- (6.,8.);
\draw (8.,9.)-- (8.,8.);
\draw (9.,8.)-- (8.,8.);
\draw (8.,7.)-- (9.,7.);
\draw (8.,7.)-- (8.,6.);
\draw (7.,7.)-- (6.,7.);
\draw (7.,7.)-- (7.,6.);
\draw (0.3404138429913868,0.3535497610962645)-- (5.960646256589087E-4,0.3537745896239297);
\draw (0.3404138429913868,0.3535497610962645)-- (0.3423142276326922,0.0038789870960769203);
\draw (0.3423142276326922,0.0038789870960769203)-- (0.3404138429913868,0.3535497610962645);
\draw (0.6735359341589903,0.35347549349725715)-- (0.6719546768124955,0.00243636257541265);
\draw (0.6735359341589903,0.35347549349725715)-- (1.,0.35189423615076243);
\draw (0.3423142276326922,0.0038789870960769203)-- (0.3404138429913868,0.3535497610962645);
\draw (0.34611499691530295,0.6576113037051228)-- (0.0011111790167362308,0.6595038251263912);
\draw (0.34611499691530295,0.6576113037051228)-- (0.34611499691530295,1.001580923781394);
\draw (0.6729811552198289,0.661412072987734)-- (0.6748815398611343,1.);
\draw (0.6729811552198289,0.661412072987734)-- (1.,0.6576113037051232);
\draw (2.33,1.)-- (2.33,0.66);
\draw (2.,0.66)-- (2.33,0.66);
\draw (2.,0.34801795723163975)-- (2.3356988669301386,0.34826178977878236);
\draw (2.3356988669301386,0.34826178977878236)-- (2.3465684383410945,0.0038642228112744335);
\draw (2.66,0.33)-- (2.66,0.);
\draw (2.66,0.33)-- (3.,0.33);
\draw (2.66,1.)-- (2.66,0.66);
\draw (2.66,0.66)-- (3.,0.66);
\draw (0.,3.)-- (0.,0.);
\draw (0.,9.)-- (0.,0.);
\begin{scriptsize}
\draw [fill=ffqqqq] (6.5,0.5) circle (1.5pt);
\draw [fill=ffqqqq] (8.5,0.5) circle (1.5pt);
\draw [fill=ffqqqq] (0.5,6.5) circle (1.5pt);
\draw [fill=ffqqqq] (2.5,6.5) circle (1.5pt);
\draw [fill=ffqqqq] (6.5,6.5) circle (1.5pt);
\draw [fill=ffqqqq] (8.5,6.5) circle (1.5pt);
\draw [fill=ffqqqq] (0.5,2.5) circle (1.5pt);
\draw [fill=ffqqqq] (2.5,2.5) circle (1.5pt);
\draw [fill=ffqqqq] (6.5,2.5) circle (1.5pt);
\draw [fill=ffqqqq] (8.5,2.5) circle (1.5pt);
\draw [fill=ffqqqq] (0.5,8.5) circle (1.5pt);
\draw [fill=ffqqqq] (2.5,8.5) circle (1.5pt);
\draw [fill=ffqqqq] (6.5,8.5) circle (1.5pt);
\draw [fill=ffqqqq] (8.5,8.5) circle (1.5pt);
\draw [fill=ffqqqq] (0.5,0.162) circle (1.5pt);
\draw [fill=ffqqqq] (0.5,0.83) circle (1.5pt);
\draw [fill=ffqqqq] (2.5,0.165) circle (1.5pt);
\draw [fill=ffqqqq] (2.5,0.825) circle (1.5pt);
\end{scriptsize}
\end{tikzpicture}
\caption{Optimal configuration of $n$ points for $16\leq n\leq 18$.}
\end{figure}
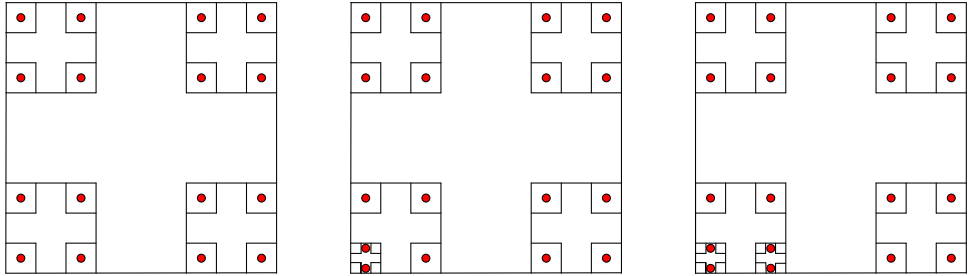

The following proposition gives all optimal sets of two-means.
\begin{prop} \label{prop1}
The sets $\set{(\frac 16, \frac 12), (\frac 56, \frac 12)}$ and  $\set{(\frac 12, \frac 16), (\frac 12, \frac 56)}$ form two different optimal sets of two-means with quantization error $V_2=\frac 5{36}.$
\end{prop}

\begin{proof}
By Lemma~\ref{lemma301}, the points in an optimal set of two-means can not lie on an oblique line of $S; $ hence, the two optimal quantizers lie either on a horizontal line or on a vertical line. We will show that these lines are $v$ or $h. $

First, assume that the optimal sets lie on a horizontal line. Let $\ga:=\set{(a, p), (b, p)}$ be an optimal set of two-means. Since the optimal quantizers are the centroids of their own Voronoi regions, it follows from by the properties of centroids that
\[(a, p) P(M((a, p)|\ga))+(b, p) P(M((b, p)|\ga))=(\frac 12, \frac 12),\]
which implies that
$$a P(M((a, p)|\ga))+b P(M((b, p)|\ga))=p P(M((a, p)|\ga))+p P(M((b, p)|\ga))=\frac 12 . $$
Thus, $p=\frac 12$, and the two optimal quantizers $(a, \frac 12)$ and $(b, \frac 12)$ lie on the line $h$ and are in opposite sides of the point $(\frac 12, \frac 12)$. Again, since the optimal quantizers are the centroids of their own Voronoi regions, it follows that $0\leq a<\frac12 <b\leq 1$. Thus,
\begin{align*}
&V_2=\int\mathop{\min}\limits_{c\in \ga} \|x-c\|^2 dP=\int\mathop{\min}\limits_{c\in \set{a, b}} \|x-(c, \frac 12)\|^2 dP\\
&=\mathop{\int}\limits_{[0, \frac {a+b}{2}]\times [0, 1]}\|x-(a, \frac 12)\|^2 dP+\mathop{\int}\limits_{[\frac {a+b}{2}, 1]\times [0, 1]}\|x-(b, \frac 12)\|^2 dP\\
&=\mathop{\int}\limits_{[0, \frac {a+b}{2}]}(x_1-a)^2 dP_c+\mathop{\int}\limits_{[\frac {a+b}{2}, 1]}(x_1-b)^2 dP_c+\Big(P_c([0, \frac {a+b}{2}])+P_c([\frac{a+b}{2}, 1]\Big) \int (x_2- \frac 12)^2 dP_c \\
&=\int\mathop{\min}\limits_{c\in \set{a, b}} (x_1-c)^2 dP_c +  \int (x_2- \frac 12)^2 dP_c.
\end{align*}
Now,
$V_2 (P_c)=\int\mathop{\min}\limits_{c\in \set{a, b}} (x-c)^2 dP_c =\frac{1}{72}$
\cite[Proposition~4.6]{GL2} and it occurs when $a=\frac 16$ and $b=\frac 56 . $
Since $\int (x_2- \frac 12)^2 dP_c=\frac {1}{8}, $ we deduce that
$$V_2=\int\mathop{\min}\limits_{c\in \set{a, b}} (x_1-c)^2 dP_c +  \int (x_2- \frac 12)^2 dP_c =
\frac{1}{72}+\frac {1}{8}=\frac 5{36}, $$
and  $\set{(\frac 16, \frac 12), (\frac 56, \frac 12)}$ is an optimal set of two-means. Due to symmetry, the set $\set{(\frac 12, \frac 16), (\frac 12, \frac 56)}$ also forms an optimal set of two-means.
\end{proof}
\medskip

We now state and prove Lemma~\ref{lemma45} and Lemma~\ref{lemma451}. These two lemmas are useful to determine the optimal sets of three means.

\begin{lemma} \label{lemma45}  The set $\ga_3=\set{(\frac 16,
\frac 1 6), (\frac 56, \frac 16), (\frac 12, \frac 56)}$ forms a CVT with three-means of distortion error $\frac 1 {12}$.
\end{lemma}

\begin{proof} Recall that the boundaries of the Voronoi regions lie along the perpendicular bisectors of the line segments joining their centers.
The perpendicular bisectors of the line segments joining each pair of points from the set $\ga_3$ 
are $SW$, $TW$ and $UW$ with equations $x_1=\frac 12$, $x_2=\frac 12 x_1+\frac 16$ and $x_2=-\frac 12 x_1+\frac 23$, respectively, and they concur at the point $W(\frac 12, \frac 5{12} )$ as shown in Figure~1$(a)$. Thus, the three regions $WUOS$, $WSAT$ and $WTBCU$ form a Voronoi tessellation of $S. $ Let us denote the three regions respectively by $M_1$, $M_2$ and $M_3$. If $(p_1, p_2)$, $(q_1, q_2)$ and $(r_1, r_2)$ are the centroids of these regions associated with $P$, respectively, we have
\begin{align*}
(p_1, p_2)&=\frac{1}{P(M_1)}\int_{M_1} x dP=\frac{1}{P(J_1)}\int_{J_1} x dP=\int x d (P\circ S_1^{-1})=S_1(\frac 12, \frac 12)=(\frac 16, \frac 1 6),\\
(q_1, q_2)&=\frac{1}{P(M_2)}\int_{M_2} x dP=\frac{1}{P(J_2)}\int_{J_2} x dP=\int x d(P\circ S_2^{-1})=S_2(\frac 12, \frac 12)=(\frac 56, \frac 1 6), \ \text{and} \\
(r_1, r_2)&=\frac{1}{P(M_3)}\int_{M_3} x dP=\frac{1}{P(J_3\uu J_4)}\int_{J_3 \uu J_4} x dP=\frac{1}{P(J_3\uu J_4)}\Big(\int_{J_3} x dP+\int_{J_4} x dP\Big)\\
&=\frac{1}{P(J_3\uu J_4)}\Big(P(J_3) \int x d(P\circ S_3^{-1})+P(J_4) \int x d(P\circ S_4^{-1})\Big)\\
&=2\Big(\frac 14  S_3(\frac 12, \frac 12)+\frac 14S_4(\frac 12, \frac 12)\Big)=(\frac 12, \frac 56).
\end{align*}
Thus, we see that the given set $\ga_3$ forms a CVT with three-means.  By \eqref{eq1}, the corresponding distortion error is
\begin{align*}
\int\min_{a \in \ga_3} &  \|x-a\|^2 dP=\int_{J_1}  \|x-(\frac 16, \frac16)\|^2 dP+\int_{J_2}  \|x-(\frac 56, \frac16)\|^2 dP+\int_{J_3\uu J_4}  \|x-(\frac 12, \frac56)\|^2 dP\\
&=\frac 1{36} V+\frac 1{36} V+\frac 1{36} \Big[ 2V +\|S_3(\frac 12, \frac 12)-(\frac 12, \frac56)\|^2+\|S_4(\frac 12, \frac 12)-(\frac 12, \frac56)\|^2\Big]=\frac{1}{12},
\end{align*}
proving the assertion.
\end{proof}
\medskip

\begin{remark} \label{rem2} The points in the set $\ga_3=\set{(\frac 16,
\frac 1 6), (\frac 56, \frac 16), (\frac 12, \frac 56)}$ given by Lemma~\ref{lemma45} form vertices of an isosceles triangle. Due to rotational symmetry and uniformity of $P,$ there are four such sets giving the same distortion error $\frac1{12}$.

\end{remark}

Figure 2(b) suggests that the set of points $\set{(\frac 56, \frac 5 6), (\frac {13}{90}, \frac {19}{30}), (\frac {19}{30}, \frac {13}{90})} $ also forms a CVT with three-means.  The following statement confirms this prediction; however, the associated distortion error is larger than $\frac 1 {12}$.

\begin{lemma} \label{lemma451}  The set $\gb_3=\set{(\frac 56,
\frac 5 6), (\frac {13}{90}, \frac {19}{30}), (\frac {19}{30}, \frac {13}{90})}$ forms a CVT with three-means and the corresponding distortion error is larger than $\frac 1 {12}$.
\end{lemma}

\begin{proof}
The perpendicular bisectors of the line segments joining each pair of points from the set of points $\set{(\frac 56,
\frac 5 6), (\frac {13}{90}, \frac {19}{30}), (\frac {19}{30}, \frac {13}{90})}$ are $SW$, $OW$ and $TW$ with equations $x_2=\frac{979}{405}-\frac{31 x_1}{9}$, $x_2=x_1$ and $x_2=\frac{979}{1395}-\frac{9 x_1}{31}$, respectively, which meet at the point $W(\frac{979}{1800}, \frac{979}{1800}) . $  Let $(p_1, p_2)$, $(q_1, q_2)$ and $(r_1, r_2)$ be the centroids of the three Voronoi regions with centers respectively $P_1 (\frac {13}{90}, \frac {19}{30})$, $Q_1 (\frac {19}{30}, \frac {13}{90})$ and $R_1 (\frac 56,
\frac 5 6)$. Since the similarity mappings preserve the ratio of the distances of a point from any other two points, by \eqref{eq67} and \eqref{eq68}, with respect to the probability measure $P$ the centroids of the triangles $OBC$ and $OAB$ are obtained as $S_1(\frac{3}{10}, \frac{7}{10})=(\frac{3}{30}, \frac{7}{30})$ and $S_1(\frac{7}{10}, \frac{3}{10})=(\frac{7}{30}, \frac{3}{30})$, respectively.
Therefore, using the definition of centroids, we have

\begin{align*}
(p_1, p_2)&=\frac{1}{P(J_3)+P(\tri OBC)}\Big(P(J_3) \int_{J_3} x dP +P(\tri OBC) \int_{\tri OBC} x dP\Big)\\
&=\frac{1}{\frac 1 4 +\frac 1 8}\Big(\frac 14 (\frac 1 6, \frac 5 6) +\frac 1 8  (\frac{3}{30}, \frac{7}{30})\Big)=\Big(\frac{13}{90},\frac{19}{30}\Big), \\
(q_1, q_2)&=\frac{1}{P(J_2)+P(\tri OAB)}\Big(P(J_2) \int_{J_2} x dP +P(\tri OAB) \int_{\tri OAB} x dP\Big)\\
&=\frac{1}{\frac 1 4 +\frac 1 8}\Big(\frac 14 (\frac 1 6, \frac 5 6) +\frac 1 8  (\frac{7}{30}, \frac{3}{30})\Big)=\Big(\frac{19}{30},\frac{13}{90}\Big),\ \text{and} \\
(r_1, r_2)&=S_4(\frac 12, \frac 12)=\Big(\frac 56, \frac 56\Big).
\end{align*}
Thus, the set $\gb_3$ forms a CVT with three-means. Now, using \eqref{eq1}, the corresponding distortion error is
\begin{align*}
\int\min_{a \in \gb_3} \|x-a\|^2 & dP
=\Big(\te{distortion error due to} (\frac{5}{6},\frac{5}{6})\Big)+2\Big(\te{distortion error due to} (\frac{13}{90},\frac{19}{30})\Big)\\
&>\frac 1{36} V+2 \Big(\int_{J_{3}}\|x-(\frac{13}{90},\frac{19}{30})\|^2 dP +\int_{J_{13}}\|x-(\frac{13}{90},\frac{19}{30})\|^2 dP\\
&+\int_{J_{113}\uu J_{143}}\|x-(\frac{13}{90},\frac{19}{30})\|^2 dP+\int_{J_{1113}\uu J_{1143}\uu J_{1413}\uu J_{1443}}\|x-(\frac{13}{90},\frac{19}{30})\|^2 dP\\
&+\int_{J_{11113}\uu J_{11143}\uu J_{11413}\uu J_{11443}\uu J_{14113}\uu J_{14143}\uu J_{14413}\uu J_{14443}}\|x-(\frac{13}{90},\frac{19}{30})\|^2 dP\\
&+\int_{J_{111113}\uu J_{111143}\uu J_{111413}\uu J_{111443}\uu J_{114113}\uu J_{114143}}\|x-(\frac{13}{90},\frac{19}{30})\|^2 dP\Big)\\
&=\frac{1247143}{14929920}>\frac 1 {12},
\end{align*}
completing the proof of the lemma.
\end{proof}

\begin{remark} \label{rem451}
One of the points of the CVT $\gb_3$ in Lemma~\ref{lemma451} is the centroid of the child $J_4$ and the other two points are equidistant from the diagonal passing through the centroid. Due to rotational symmetry of $S, $ there are four such CVTs with three-means in which one point is the centroid of one of the children $J_1$, $J_2$, $J_3$ or $J_4$ and the other two points are equidistant from the diagonal passing through the centroid, and all have the same distortion error larger than $\frac 1{12}$.
\end{remark}

\begin{prop} \label{prop2}
 Let $\ga_3$ be the set given by Lemma~\ref{lemma45}. Then, $\ga_3$ forms an optimal set of three-means with quantization error $\frac 1{12}$. The number of optimal sets of three-means is four.
\end{prop}

\begin{proof}
As mentioned in Remark~\ref{rem1}, the Cantor dust has four lines of maximum symmetry $h$, $v$, $r$ and $l$ (see Figure~\ref{Fig1}). Due to this, we conjecture that one of the points in an optimal set of three-means must lie on any of the above four lines, and other two will be equidistant from the line. Thus, comparing the distortion errors for the CVTs given by Lemma~\ref{lemma45} and Lemma~\ref{lemma451}, we deduce that the CVT $\ga_3$, given by Lemma~\ref{lemma45}, forms an optimal set of three-means with quantization error $\frac 1{12}$. Remark~\ref{rem2} implies that the number of such sets is four.
\end{proof}

\begin{remark}Lemma~\ref{lemma45} and Lemma~\ref{lemma451} together show that under squared error distortion measure, the centroid condition is not sufficient for optimal quantization for singular continuous probability measures on $\D R^2$, which is already known for absolutely continuous probability measures on $\D R^2$ \cite {DFG}, and for singular continuous probability measure on $\D R$ \cite{R2}.
\end{remark}


\begin{lemma} \label{lemma30}
 Let $n\geq 4$ and let $\ga_n$ be an optimal set of $n$-means such that $\ga_n\ii J_i\neq \es$ for $1\leq i\leq 4$, and $\ga_n$ does not contain any point from $J\setminus \uu_{i=1}^4 J_i$.  If $\gb_i:=\ga_n \ii J_i$ with $n_i:=\te{card}(\gb_i)$, then $S_i^{-1}(\gb_i)$ is an optimal set of $n_i$-means. Moreover,
\[V_n=\frac {1}{36} \left(V_{n_1}+V_{n_2}+V_{n_3}+V_{n_4}\right).\]
\end{lemma}

\begin{proof} By the hypothesis, $\gb_i \neq \emptyset ,$ for all $1\leq i\leq 4 , $ and $\ga_n$ does not contain any point from $J\setminus \uu_{i=1}^4 J_i;$ hence, $  \ga_n=\uu_{i=1}^4 \gb_i$.   Since $\ga_n$  is an optimal set of $n$-means,
\begin{equation*} \label{eq46}
V_n=\sum_{i=1}^4\int_{J_i} \min_{a\in \ga_n} \|x-a\|^2 dP=\sum_{i=1}^4\int_{J_i} \min_{a\in \gb_i} \|x-a\|^2 dP.
\end{equation*}
Now, using Lemma~\ref{lemma1} we have
\begin{equation}\label{eq47}
V_n=\frac{1}{36} \sum_{i=1}^4\int \min_{a\in \gb_i} \|x-S_i^{-1}(a)\|^2 dP=\frac{1}{36} \sum_{i=1}^4\int \min_{a\in S_i^{-1}(\gb_i)} \|x-a\|^2 dP.
\end{equation}
If $S_1^{-1}(\gb_1)$ is not an optimal set of $n_1$-means, then we can find a set $\gg_1\sci \D R^2$ with card$(\gg_1)=n_1$ such that
\[\int \min_{a\in \gg_1} \|x-a\|^2 dP <\int \min_{a\in S_1^{-1}(\gb_1)} \|x-a\|^2 dP.\]
But, then $S_1(\gg_1)\uu \gb_2\uu \gb_3\uu\gb_4$ will be a set of cardinality $n$, and
\begin{align*}
&\int \min\set{ \|x-a\|^2 : a\in S_1(\gg_1)\uu \gb_2\uu \gb_3\uu\gb_4}dP\\
&=\int_{J_1} \min_{a\in S_1(\gg_1)} \|x-a\|^2 dP+\frac{1}{36} \sum_{i=2}^4\int \min_{a\in S_i^{-1}(\gb_i)} \|x-a\|^2 dP\\
&=\frac 1 {36} \int \min_{a\in S_1(\gg_1)} \|x-S_1^{-1}(a)\|^2 dP+\frac{1}{36} \sum_{i=2}^4\int \min_{a\in S_i^{-1}(\gb_i)} \|x-a\|^2 dP\\
&=\frac 1 {36} \int \min_{a\in \gg_1} \|x-a\|^2 dP+\frac{1}{36} \sum_{i=2}^4\int \min_{a\in S_i^{-1}(\gb_i)} \|x-a\|^2 dP\\
&<\frac 1 {36} \int \min_{a\in S_1^{-1}(\gb_1)} \|x-a\|^2 dP+\frac{1}{36} \sum_{i=2}^4\int \min_{a\in S_i^{-1}(\gb_i)} \|x-a\|^2 dP.
\end{align*}
Thus by \eqref{eq47}, we have $\int \min\set{ \|x-a\|^2 : a\in S_1(\gg_1)\uu \gb_2\uu \gb_3\uu\gb_4}dP<V_n$, which contradicts the fact that $\ga_n$ is an optimal set of $n$-means, and so $S_1^{-1}(\gb_1)$ is an optimal set of $n_1$-means. Similarly, one can show that $S_i^{-1}(\gb_i)$ are optimal sets of $n_i$-means for all $2\leq i\leq 4$. Thus, \eqref{eq47} implies $V_n=\frac {1}{36} \left(V_{n_1}+V_{n_2}+V_{n_3}+V_{n_4}\right)$. This completes the proof of the lemma.
\end{proof}

The similarity mappings $S_i$ preserve the ratio of the distances of a point from any other two points; hence, the following proposition is an immediate corollary of Lemma~\ref{lemma30}.

\begin{prop} \label{prop3}
 Let $n\geq 4$ and $\ga_n$ be an optimal set of $n$-means such that $\ga_n\ii J_\gs\neq \es$ for $\gs\in I^{\ell(n)}$ for some $\ell(n) \in \D N$, and $\ga_n$ does not contain any point from $J\setminus \uu_{\gs\in I^{\ell(n)}} J_\gs$.  If $\gb_\gs:=\ga_n \ii J_\gs$ with $n_\gs:=\te{card}(\gb_\gs)$, then $S_\gs^{-1}(\gb_\gs)$ is an optimal set of $n_\gs$-means and
\[V_n=\frac {1}{36^{\ell(n)}} \sum_{\gs \in I^{\ell(n)}}V_{n_\gs}.\]
\end{prop}

\medskip

Proposition \ref{prop3} provides a major step in describing the optimal sets of $n$-means and the associated quantization errors, subject to the conditions that $\ga_n\ii J_\gs\neq \es$ and $ \ga_n \cap (J\setminus \uu_{\gs\in I^{\ell(n)}} J_\gs) =\emptyset . $  Thus, it remains to prove that these conditions are satisfied, which will be the focus of the next three statements.

\begin{lemma}\label{lemma55}
If $\ga_n$ is an optimal set of $n$-means for $n\geq 4 ,$ then $\ga_n\ii J_i\neq\es$ for $1\leq i\leq 4$.
\end{lemma}
\begin{proof}  Recall that the Cantor dust has four lines of symmetry: the lines $h,\ v, \ r$ and $l. $  First, we will prove that, for $n\geq 4, $ an optimal set of $n$-means $\ga_n$ meets each of the quadrants determined by the lines $h$ and $v. $  Consider the set $\gb:=\set{S_i(\frac 12, \frac 12) : 1\leq i\leq 4}$. Then, the distortion error due to the set $\gb$ is given by
\[\int\min_{a \in \gb} \|x-a\|^2dP=\sum_{i=1}^4 \int_{J_i}\|x-S_i(\frac 12, \frac 12)\|^2dP=\frac 1{9} V=\frac 1{36}.\]
Since $V_n$ is the quantization error for $n$-means for $n\geq 4$, we have $V_n\leq V_4\leq \frac 1{36} . $  If all the elements of $\ga_n$ lie on the line $h, $  then, for any $x\in \uu_{i=1}^4 J_{ii}$, $\min_{a \in \ga_n}\|x-a\|^2 \geq (\frac 12-\frac 19)^2=\frac{49}{324}$, and the distortion error is
\begin{align*}
\int\min_{a \in \ga_n}\|x-a\|^2dP> \sum_{i=1}^4 \int_{J_{ii}}\min_{a \in \ga_n}\|x-a\|^2dP\geq \sum_{i=1}^4 \frac{49}{324} P(J_{ii})=\frac{49}{1296}> \frac 1{36}\geq V_n,
\end{align*}
which contradicts the optimality of $\alpha . $ Therefore, $\ga_n \nsubseteq h .$  It follows by the same argument that $\ga_n \nsubseteq v .$
If $\ga_n \subset r, $ then, for any $x \in J_{22}\uu J_{33}, $ $ \ \min_{a \in \ga_n}\|x-a\|^2 \geq \|(\frac 19, \frac 89)-(\frac 12, \frac 12)\|^2=\frac{49}{162}. $  Hence,
\begin{align*}
\int\min_{a \in \ga_n}\|x-a\|^2dP>\int_{J_{22}\uu J_{33}}\min_{a \in \ga_n}\|x-a\|^2dP\geq 2 \cdot\frac{49}{162}\cdot \frac 1{16}=\frac{49}{1296}>V_n,
\end{align*}
which, again, contradicts optimality.  Similarly,  $\ga_n \nsubseteq l, $ either.  Now, by the properties of centroids, we have
\[\sum_{i=1}^4(a_i, b_i) P(M((a_i, b_i)|\ga_n))=(\frac 12, \frac 12),\]
yielding that $\sum_{i=1}^4a_iP(M((a_i, b_i)|\ga_n))=\frac 12$ and $\sum_{i=1}^4 b_i P(M((a_i, b_i)|\ga_n))=\frac 12 . $  This implies that all the elements of $\ga_n$ cannot lie only on one side of the line $h $ nor only on one side of the line $v$.

Next, we claim that at least a pair of points of $\ga_n$ lie in the region above $h$ and at least another pair lie in the region below of $h. $  Suppose that there is only one point of $\ga_n$ that lies above the line $h$. Due to symmetry we can assume that the point lies on the line $v$. Then, for $x\in J_{33}\uu J_{44}$, $\min_{a \in \ga_n}\|x-a\|^2 \geq (\frac 12-\frac 19)^2=\frac{49}{324}$, and for $x\in A=J_{31}\uu J_{32}\uu J_{34}\uu J_{41}\uu J_{42}\uu J_{43}$, $\min_{a \in \ga_n}\|x-a\|^2 \geq (\frac 12-\frac 13)^2=\frac 1{36} . $  Then,
\begin{align*}
\int\min_{a \in \ga_n}\|x-a\|^2dP & >\int_{J_{33}\uu J_{44}}\min_{a \in \ga_n}\|x-a\|^2dP+\int_{A}\min_{a \in \ga_n}\|x-a\|^2dP\\
&\geq 2 \cdot\frac{49}{324}\cdot \frac 1{16}+6\cdot \frac 1{36} \cdot \frac 1{16}=\frac{19}{648}>V_n,
\end{align*}
which is a contradiction, proving the claim.  It follows in the same manner that at least two of the points of $\ga_n$ lie on one side of the line $v$ and at least two of the points of $\ga_n$ lie on the other side the line $v . $ Therefore, $\ga_n$ contains points from each of the four quadrants. Since $ supp(P) \subset J_1\uu J_2\uu J_3\uu J_4$ and $P$ is symmetrically distributed over $J$, we deduce that $\ga_n$ contains points from each $J_i,\ 1\leq i\leq 4; $ i.e., $\ga_n\ii J_i\neq \es$ for $1\leq i\leq 4$.
\end{proof}
\medskip

Below, for any $1\leq i, j\leq 4 , $ we will use the notation $ij^{(k)}$ for the concatenation of $i$ with $j^{(k)}$, where $j^{(k)}$ denotes the $k$-times concatenation of $j$ with itself.

\begin{lemma}\label{lemma56}
If $n\geq 4$ and $\ga_n$ is an optimal set of $n$-means, then $\ga_n\ii (J\setminus \mathop \uu\limits_{i=1}^4 J_i)=\emptyset . $
\end{lemma}

\begin{proof}
First, assume that $n=m 4^{\ell(n)}$, where $m, \ell(n)\in \D N$. Then, due to the maximum symmetry (i.e., symmetry with respect to $h,l,r,v$ and $P$) and Lemma~\ref{lemma55}, the optimal set $\ga_n$ contains $m 4^{\ell(n)-1}$ elements from each of $J_i, \ 1\leq i \leq 4 . $  By Proposition 1.1, the Voronoi region of each point in an optimal set of n-means has positive probability measure.  Since support of $P$ lies in $ \cup_{i=1}^4 J_i, $ it follows that $\ga_n\ii (J\setminus \mathop \uu\limits_{i=1}^4 J_i)=\emptyset . $   Thus, it remains to prove the assertion for $n=m4^{\ell(n)}+k, \ 1\leq k \leq 3. $

Suppose $\ga_n\ii (J\setminus \mathop \uu\limits_{i=1}^4 J_i) \neq \emptyset ,\  n=m4^{\ell(n)}+k, \ 1\leq k \leq 3. $ Again, due to the maximum symmetry, without loss of generality, we can assume that
\begin{itemize}
\item[(i)] when $k=1,$ $\ga_n$ contains $m 4^{\ell(n)-1}$ elements from each of $J_i$, and the remaining one element is the center $(\frac 12, \frac 12);$
\item[(ii)] when $k=2,$ $\ga_n$ contains $m 4^{\ell(n)-1}$ elements from each $J_i$, and the remaining two elements are the points $(\frac 12, \frac 16)$ and $(\frac 12, \frac 56) ; $ and
\item[(iii)] when $k=3, $ $\ga_n$ contains $m 4^{\ell(n)-1}+1$ elements from each $J_1$ and $J_2$, and $m 4^{\ell(n)-1}$ elements from each of $J_3$ and $J_4$, and the remaining one element is $(\frac 12, \frac 56).$
\end{itemize}
In all these cases, if the Voronoi regions $M((\centerdot , \frac 12)|\ga_n)$ of these points do not contain any point from $\mathop{\uu}\limits_{i=1}^4 J_i$,  then $P(M((\centerdot , \frac 12)|\ga_n))=0, $ and this contradicts Proposition 1.1.

In case (i), if $M((\frac 12, \frac 12)|\ga_n)$ contains points from $J_i$ for all $1\leq i\leq 4, $. then there is a positive integer $k$ such that $J_{14^{(k)}}\uu J_{23^{(k)}}\uu J_{32^{(k)}}\uu J_{41^{(k)} }\sci M((\frac 12, \frac 12)|\ga_n) . $ Then, the distortion error due to the set $\ga_n$ is larger than the distortion error due to the set $\ga_n$ when the extra one point is moved to any of the children $J_i, \ 1\leq i\leq 4 . $  This contradicts the fact that $\ga_n$ is an optimal set of $n$-means. Thus, $\ga_n\ii (J\setminus \mathop \uu\limits_{i=1}^4 J_i) =\emptyset . $

In case (ii), if $M((\frac 12, \frac 16)|\ga_n)$ and  $M((\frac 12, \frac 56)|\ga_n)$ contain points from $\mathop{\uu}\limits_{i=1}^4 J_i , $ then $M((\frac 12, \frac 16)|\ga_n)$ will contain points from both $J_1$ and $J_2$, and $M((\frac 12, \frac 56)|\ga_n)$ will contain points from both $J_3$ and $J_4$. But, in that situation the distortion error due to the set $\ga_n$ is larger than the distortion error due to the set $\ga_n$ when the point $(\frac 12, \frac 16)$ is moved to either $J_1$ or $J_2$, and the point $(\frac 12, \frac 56)$ is moved to either $J_3$ or $J_4, $ contradicting the optimality of $\ga_n . $ Thus, $\ga_n\ii (J\setminus \mathop \uu\limits_{i=1}^4 J_i) = \emptyset . $

In case (iii), if $M((\frac 12, \frac 56)|\ga_n)$ contains points from $\mathop{\uu}\limits_{i=1}^4 J_i, $ then $M((\frac 12, \frac 56)|\ga_n)$ will contain points from both $J_3$ and $J_4$. But, then the distortion error due to the set $\ga_n$ is larger than the distortion error due to the set $\ga_n$ when the point $(\frac 12, \frac 56)$ is moved to either $J_3$ or $J_4, $ contradicting the optimality of $\ga_n . $  Thus, $\ga_n\ii (J\setminus \mathop \uu\limits_{i=1}^4 J_i) =\emptyset . $
\end{proof}

In the next proposition, $\gs^-$ will denote the word $\gs_1\gs_2\cdots\gs_{n-1}, $ where $\gs=\gs_1\gs_2\cdots \gs_n . $

\begin{prop}\label{prop32}
 Let $\ga_n$ be an optimal set of $n$-means for $n\geq 4^{\ell(n)}$ for some $\ell(n)\in \mathbb N$. Then, for each $\gs \in I^{\ell(n)}, \ \ga_n\ii J_\gs\neq\es $ and $ \ga_n\ii (J_{\gs^-} \setminus \mathop \uu\limits_{i=1}^4 J_{\gs^- i})= \emptyset . $
\end{prop}
\begin{proof}
First, we will consider the case $n=4^{\ell(n)}$ for some $\ell(n)\in \D N$.  If $\ell(n)=1, \ n=4, $ then the assertion follows from Lemma~\ref{lemma55} and Lemma~\ref{lemma56}.  So, let $\ell(n)\geq 2$. Then, by Lemma~\ref{lemma55} and Lemma~\ref{lemma56}, and the symmetry of $P$, the set $\ga_n$ contains $4^{\ell(n)-1}$ elements from each of $J_{i_1}$, where $1\leq i_1\leq 4$.  Since the probability measure $P\circ S_{i_1}^{-1}$ on $J_{i_1}$ is also symmetric as the probability measure $P$ on $J, $ by applying Lemma~\ref{lemma55} and Lemma~\ref{lemma56} on $J_{i_1}$, we see that $\ga_n$ contains $4^{\ell(n)-2}$ elements from each of $J_{i_1i_2}$, where $1\leq i_2\leq 4$. Proceeding in this way inductively, we see that $\ga_n$ contains one element from each of $J_\gs$, where $\gs \in I^{\ell(n)}$, in other words, $\ga_n\ii J_\gs\neq \es$ for $\gs \in I^{\ell(n)}$. Since $\ga_n\ii J_{\gs^-}=\uu_{i=1}^4 \ga_n\ii J_{\gs^-i}$ it follows that $\ga_n\ii (J_{\gs^-} \setminus \mathop \uu\limits_{i=1}^4 J_{\gs^- i})=\emptyset . $

Now, let $n=4^{\ell(n)}+1$ for some $\ell(n) \in \D N$.  Again, by Lemma~\ref{lemma55}, Lemma~\ref{lemma56} and symmetry of $P$, there exists an element $i_1\in I$ such that $\ga_n$ contains $4^{\ell(n)-1}+1$ elements from $J_{i_1}$. Applying Lemma~\ref{lemma55} and Lemma~\ref{lemma56} on $J_{i_1}$ again, and proceeding in this way inductively, we can see that there exists a word $\gs\in I^{\ell(n)}$, where $\gs$ is an extension of $i_1$, such that $\ga_n$ contains two elements from $J_\gs$, and $\ga_n$ contains only one element from each  $J_\gt$ for $\gt\in I^{\ell(n)}$ with $\gt\neq \gs$. Thus, the proposition is true for $n=4^{\ell(n)}+1$.

When $n=4^{\ell(n)}+2$ for some $\ell(n) \in \D N$, via Lemma~\ref{lemma55}, Lemma~\ref{lemma56} and symmetry of $P$, there exists two elements $i_1, i_2\in I$, such that $\ga_n$ contains $4^{\ell(n)-1}+1$ elements from each of $J_{i_1}$ and $J_{i_2}$, and $4^{\ell(n)-1}$ elements from each of $J_j$ for $j\in I\setminus\set{i_1, i_2}$. Applying Lemma~\ref{lemma55} and Lemma~\ref{lemma56} on both $J_{i_1}$ and $J_{i_2}$ again, and proceeding in this way inductively, we can see that there exist words $\gs, \gt \in I^{\ell(n)}$, where $\gs$ is an extension of $i_1$, and $\gt$ is an extension of $i_2$, such that $\ga_n$ contains two elements from each of $J_\gs$ and $J_\gt$, and $\ga_n$ contains only one element from each  $J_\gd$ for $\gd\in I^{\ell(n)}\setminus\set{\gs, \gt}$. Thus, the proposition is true for $n=4^{\ell(n)}+2$.

If $n=4^{\ell(n)}+3,$ the assertion is obtained similarly to the case $n=4^{\ell(n)}+1.$  Thus, we deduce that the statement is valid for any $n\geq 4^{\ell(n)},\ \ell(n)\in \D N$.

\end{proof}

By Lemma~\ref{lemma333}, the set $\ga_1=\{(\frac12, \frac12)\} $ is the only optimal set of one-mean.
with quantization error $V=\frac 14$.  By
Proposition~\ref{prop1} and Proposition~\ref{prop2}, the sets
$ \ga_2=\{(\frac12, \frac16), (\frac12, \frac56)\}$ and
$ \ga_3=\{(\frac16, \frac16), (\frac56, \frac16), (\frac12, \frac56)\} $
are optimal sets of two- and three-means
with quantization error $\frac{5}{36} $ and $ \frac{1}{12}, $ respectively.
The sets $\ga_2 $ and $\ga_3 $ are not the only optimal sets of two- and three-means; indeed, the total number of optimal sets of two-means is two and the total number of optimal sets of three-means is four.  With this, the optimal sets of $n$-means for all $n \geq 4$, their numbers and the associated quantization error are determined by the following theorem.

\begin{theorem}  \label{Th1}
Let $P$ be a Borel probability measure on $\D R^2$ supported by the Cantor dust $S. $  Let $n\geq 4, \ 1\leq m\leq 3$. Then,
\begin{itemize}
\item[(i)] if $n=m 4^{\ell(n)} $ for some $\ell(n) \in \mathbb Z^+$, then
$\ga_n=\set{S_\gs(\ga_m) : \gs \in I^{\ell(n)}}$ is an optimal set of $n$-means. The number of such sets is $(2^{m-1})^{4^{\ell(n)}}$ and the corresponding quantization error is
\[V_n=\sum_{\gs \in I^{\ell(n)}} \int_{J_\gs} \min_{a \in S_\gs(\ga_m)} \|x-a\|^2dP.\]
\item[(ii)] if $n=m 4^{\ell(n)}+k $, where $1\leq k<4^{\ell(n)}$ for some $\ell(n)\in \mathbb Z^+ , $ and $t\sci I^{\ell(n)}$ with card$(t)=k$, then,
$ \ga_n(t)=\set{S_\gs(\ga_m) : \gs \in I^{\ell(n)}\setminus t} \uu\set {S_{\gs}(\ga_{m+1}) : \gs \in  t} $
is an optimal set of $n$-means. The number of such sets is $(2^{m-1})^{4^{\ell(n)}-k}\cdot {}^{4^{\ell(n)}}C_k\cdot 2^{mk}$ if $m=1, 2$, and $(2^{m-1})^{4^{\ell(n)}-k}\cdot {}^{4^{\ell(n)}}C_k$ if $m=3$; the corresponding quantization error is
\[V_n=\sum_{\gs \in I^{\ell(n)}\setminus t} \int_{J_\gs}\min_{a \in S_\gs(\ga_m)} \|x-a\|^2dP+\sum_{\gs \in t}\int_{J_{\gs}}\min_{a \in S_\gs(\ga_{m+1})} \|x-a\|^2dP,\]
where $^uC_v =\begin{pmatrix} u \\ v \end{pmatrix} , $ the binomial coefficients.
\end{itemize}
\end{theorem}

\begin{proof} Let $m=1, 2, 3$. Let $n=m 4^{\ell(n)},\ 1\leq m \leq 3 , $ for some $\ell(n) \in \D N$. Then, by Proposition~\ref{prop32}, it follows that $\ga_n$ contains $m$ elements from each $J_\gs$ for $\gs\in I^{\ell(n)}$, and Proposition~\ref{prop3} implies that $S_\gs^{-1}(\ga_n\ii J_\gs)$ is an optimal set of $m$-means (i.e.,  $\ga_n\ii J_\gs=S_\gs(\ga_m)$); thus,
\[\ga_n=\mathop{\uu}\limits_{\gs \in I^{\ell(n)}} S_\gs(\ga_m)=\set{S_{\gs}(\ga_m) : \gs \in I^{\ell(n)}}.\]
Since $\ga_m$ can be chosen in $2^{m-1}$ different ways, the number of such sets is $(2^{m-1})^{4^{\ell(n)}}$, and the corresponding quantization error is given by
\[V_n=\int \min_{a \in \ga_n} \|x-a\|^2dP=\sum_{\gs \in I^{\ell(n)}} \int_{J_\gs} \min_{a \in \ga_n} \|x-a\|^2dP=\sum_{\gs \in I^{\ell(n)}} \int_{J_\gs} \min_{a \in  S_\gs(\ga_m)} \|x-a\|^2dP.\]

To prove $(ii)$ we proceed as follows: let $t\sci I^{\ell(n)}$ with card$(t)=k$. Then, by Proposition~\ref{prop32}, it follows that $\ga_n$ contains $m$ elements from $J_\gs$ for each $\gs\in I^{\ell(n)}\setminus t$, and $(m+1)$ elements from $J_\gs$ for $\gs \in t$. In other words, $\ga_n\ii J_\gs=S_\gs(\ga_m)$ for $\gs\in I^{\ell(n)}\setminus t$, and $\ga_n\ii J_\gs=S_\gs(\ga_{m+1})$ for $\gs \in t$. Thus,
\[\ga_n(t)=\set{S_\gs(\ga_m) : \gs \in I^{\ell(n)}\setminus t} \uu\set {S_{\gs}(\ga_{m+1}) : \gs \in  t}.\]
The corresponding quantization error is given by
\begin{align*}
&V_n=\int \min_{a \in \ga_n} \|x-a\|^2dP=\sum_{\gs \in I^{\ell(n)}\setminus t} \int_{J_\gs} \min_{a \in \ga_n} \|x-a\|^2dP+\sum_{\gs \in t}\int_{J_{\gs}} \min_{a \in \ga_n} \|x-a\|^2dP\\
&=\sum_{\gs \in I^{\ell(n)}\setminus t} \int_{J_\gs}\min_{a \in S_\gs(\ga_m)} \|x-a\|^2dP+\sum_{\gs \in t}\int_{J_{\gs}}\min_{a \in S_\gs(\ga_{m+1})} \|x-a\|^2dP.
\end{align*}
Recall that $\ga_2$ can be chosen in two different ways, $\ga_3$ can be chosen in three different ways, and $\ga_4$ can be chosen in only one way. Thus, if $m=1$, the number of $\ga_n$ is $^{4^{\ell(n)}}C_k\cdot 2^k$; if $m=2$, the number of $\ga_n$ is $2^{4^{\ell(n)}-k}\cdot {}^{4^{\ell(n)}}C_k\cdot 4^k$; and if $m=3$, the number of $\ga_n$ is $4^{4^{\ell(n)}-k}\cdot {}^{4^{\ell(n)}}C_k$. Hence, the proof of the theorem is complete.
\end{proof}
\bigskip

\section{Quantization dimension and quantization coefficient}

Since the standard Cantor dust $S$ under investigation satisfies the strong separation condition, with each $S_i$ having contracting factor of $\frac{1}{3}, $ it's Hausdorff dimension is equal to the similarity dimension.  Hence, from the equation $4 (\frac 1 3)^\gb=1, $ we have $\dim_{\te{H}}(S)=\beta =\frac{\log 4}{\log 3}. $  By Theorem~14.17 in \cite{GL1}, the quantization dimension $D_2 (P) $ exists and is equal to $ \beta. $

Next, we proceed to show that $\gb$ dimensional quantization coefficient for $P$ does not exist. The proof of it follows the same general strategy as \cite[Theorem~6.3]{GL2}.  We begin with the following simple observation.
If $f : [1, 2]\to \D R$ be a function defined by $f(x)=\frac 1{36} x^{\frac 2\gb}(13-x), $ then it is strictly increasing on $[1,2]; $ and  hence $f([1, 2])=[\frac {1}{3}, \frac {11}{12}] . $

\begin{theorem} \label{Th3}
 $\gb$-dimensional quantization coefficient for $P$ does not exist.
\end{theorem}
\begin{proof} Let $(n_k)_{k\in \D N}$ be a subsequence of the set of natural numbers such that $4^{\ell(n_k)}\leq n_k<2\cdot 4^{\ell(n_k)}$. The assertion of the theorem will follow if we show that the set of accumulation points of the sequence $(n_k^{\frac 2\gb} V_{n_k})_{k\geq 1}$ is $[\frac {1}{3}, \frac {11}{12}]$. Let  $y \in [\frac {1}{3}, \frac {11}{12}], $ then $y=f(x)$ for some $x\in [1, 2]$. Set $n_{k_\ell}=\lfloor x 4^{\ell}\rfloor$, where $\lfloor x 4^{\ell}\rfloor$ denotes the greatest integer less than or equal to $ x 4^{\ell}$. Then, $n_{k_\ell}<n_{k_{\ell+1}}$ and $\ell(n_{k_\ell})=\ell$, and there exists $x_{k_\ell} \in [1, 2]$ such that $n_{k_\ell}=x_{k_\ell} 4^\ell$. Recall that by $\ell(n_{k_\ell})=\ell$ it is meant that $4^\ell\leq n_{k_\ell}<4^{\ell+1}$. Notice that if $4^{\ell(n)}\leq n\leq 4^{\ell(n)+1}$, then by Theorem~\ref{Th1}, we have
\begin{align*}
V_n=(2 \cdot 4^{\ell(n)}-n) \frac 1{36^{\ell(n)}}\frac 14+(n-4^{\ell(n)}) \frac 1{36^{\ell(n)}} \frac 5{36}=\frac 1{36^{\ell(n)+1}} (13 \cdot 4^{\ell(n)}-4n).
\end{align*}
Thus, putting the values of $n_{k_\ell}$ and $V_{n_{k_\ell}}$, we obtain
\begin{align*}
n_{k_\ell}^{\frac 2\gb} V_{n_{k_\ell}}=n_{k_\ell}^{\frac 2\gb}  \frac 1{36^{\ell+1}} (13 \cdot 4^{\ell}-4n_k)=x_{k_\ell}^{\frac 2\gb} 4^{\frac 2 \gb} \frac 1{36^{\ell+1}} (13 \cdot 4^{\ell}-4 x_{k_\ell} 4^\ell)=x_{k_\ell}^{\frac 2\gb} 9^\ell \frac 1{36^{\ell+1}} (13 \cdot 4^{\ell}-4 x_{k_\ell} 4^\ell),
\end{align*}
which yields
\begin{align} \label{eq45} n_{k_\ell}^{\frac 2\gb} V_{n_{k_\ell}}=\frac1{36}x_{k_\ell}^{\frac 2\gb} (13-4x_{k_\ell})=f(x_{k_\ell}).\end{align}
Again, $x_{k_\ell} 4^{\ell}\leq x 4^\ell<x_{k_\ell} 4^{\ell}+1$, which implies $x-\frac 1{4^{\ell}}< x_{k_\ell} \leq x$, and so, $\mathop{\lim}\limits_{\ell\to \infty} x_{k_\ell}=x$. Since, $f$ is continuous, we have
\[\mathop{\lim}\limits_{\ell\to \infty}  n_{k_\ell}^{\frac 2\gb} V_{n_{k_\ell}}=f(x)=y,\]
which yields the fact that $y$ is an accumulation point of the subsequence $(n_k^{\frac 2\gb}  V_{n_k})_{k\geq 1}$ whenever $y\in [\frac {1}{3}, \frac {11}{12}]$. To prove the converse, let $y$ be an accumulation point of the sequence  $(n_k^{\frac 2 \gb} V_{n_k})_{k\geq 1}$. Then, there exists a subsequence  $(n_{k_i}^{\frac 2\gb} V_{n_{k_i}})_{i\geq 1}$ of  $(n_k^{\frac 2\gb} V_{n_k})_{k\geq 1}$ such that $\mathop{\lim}\limits_{i\to \infty}n_{k_i}^{\frac 2\gb} V_{n_{k_i}}=y$. Set $\ell_{k_i}=\ell(n_{k_i})$ and $x_{k_i}=\frac{n_{k_i}}{4^{\ell_{k_i}}}$. Then,  $x_{k_i} \in[1, 2]$, and as shown in \eqref{eq45}, we have
\[n_{k_i}^{\frac 2\gb} V_{n_{k_i}}=f(x_{k_i}).\]
Let $(x_{k_{i_j}})_{j\geq 1}$ be a convergent subsequence of $(x_{k_i})_{i\geq 1}$, then we obtain
\[y=\lim_{i\to \infty} n_{k_i}^{\frac 2\gb} V_{n_{k_i}}=\lim_{j\to \infty}n_{k_{i_j}}^{\frac 2\gb}  V_{n_{k_{i_j}}}=\lim_{j\to \infty}f(x_{k_{i_j}}) \in [\frac {1}{3}, \frac {11}{12}].\]
Thus, the set of accumulation points of the sequence $(n_k^{\frac 2\gb} V_{n_k})_{k\geq 1}$ is $[\frac {1}{3}, \frac {11}{12}]$, which yields the proof of the theorem.
\end{proof}


\bigskip



\begin{thebibliography}{9999}
\bibitem {AW} E.F. Abaya and G.L. Wise, \emph{Some remarks on the existence of optimal quantizers}, Statistics \& Probability Letters, 2(6) (1984), 349-351.

\bibitem {DFG} Q. Du, V. Faber and M. Gunzburger, \emph{Centroidal Voronoi Tessellations: Applications and Algorithms}, SIAM Review, 41(4) (1999), 637-676.

\bibitem {DR} C.P. Dettmann and M.K. Roychowdhury, \emph{Quantization for uniform distributions on equilateral triangles}, Real Analysis Exchange, 42(1) (2017), 149-166.


\bibitem {CR} D. \c C\"omez  and M.K. Roychowdhury, \emph{Quantization for uniform distributions on stretched Sierpi\'nski triangles}, Monatshefte f\"ur Mathematik, Volume 190(1) (2019), 79-100.

\bibitem {GG} A. Gersho and R.M. Gray, \emph{Vector quantization and signal compression}, Kluwer Academy publishers: Boston, 1992.

\bibitem {GKL}  R.M. Gray, J.C. Kieffer and Y. Linde, \emph{Locally optimal block quantizer design}, Information and Control, 45 (1980), 178-198.


\bibitem {GL} A. Gy\"orgy and T. Linder, \emph{On the structure of optimal entropy-constrained scalar quantizers}, IEEE Transactions on Information Theory, 48(2) (2002), 416-427.

\bibitem {GL1} S. Graf and H. Luschgy, \emph{Foundations of quantization for probability distributions}, Lecture Notes in Mathematics 1730, Springer, Berlin, 2000.


\bibitem {GL2} S. Graf and H. Luschgy, \emph{The Quantization of the Cantor Distribution}, Math. Nachr., 183 (1997), 113-133.

  \bibitem {GN} R.M. Gray and D.L. Neuhoff, \emph{Quantization}, IEEE Transactions on Information Theory, 44(6) (1998), 2325-2383.


\bibitem {H} J. Hutchinson, \emph{Fractals and self-similarity}, Indiana Univ. J., 30 (1981), 713-747.


 \bibitem{R1} M.K. Roychowdhury, \emph{Optimal quantizers for some absolutely continuous probability measures}, Real Analysis Exchange, Vol. 43(1), 2017, pp. 105-136.
\bibitem{R2} M.K. Roychowdhury, \emph{Quantization and centroidal Voronoi tessellations for probability measures on dyadic Cantor sets}, Journal of Fractal Geometry, 4 (2017), 127-146.


\bibitem{RR} J. Rosenblatt and M.K. Roychowdhury, \emph{Optimal quantization for piecewise uniform distributions}, Uniform Distribution Theory 13 (2018), no. 2, 23-55.


\bibitem {Z} R. Zam, \emph{Lattice Coding for Signals and Networks: A Structured Coding Approach to Quantization, Modulation, and Multiuser Information Theory}, Cambridge University Press, 2014.

\end{thebibliography}
\end{document}